\documentclass[12pt]{amsart}
\title{Linear forms and higher-degree uniformity for functions on $\F_p^n$}
\author{W.T. Gowers}
\address{University of Cambridge, Department of Pure Mathematics and Mathematical Statistics, 
Wilberforce Road, Cambridge CB3 0WB, UK.}
\email{w.t.gowers@dpmms.cam.ac.uk}
\author{J. Wolf}
\address{Rutgers The State University of New Jersey, Department of Mathematics, 110 Frelinghuysen Rd., Piscataway, NJ 08854, U.S.A.}
\email{julia.wolf@cantab.net}
\thanks{Both authors gratefully acknowledge the hospitality of the Mathematical Sciences Research Institute, Berkeley, where important parts of this work were carried out.}

\usepackage{amsmath, wrapfig}
\usepackage{dsfont, a4wide, amsthm, amssymb, amsfonts, graphicx}
\usepackage{fancyhdr, xspace, psfrag, setspace, supertabular, color}

\begin{document}

\newtheorem{theorem}{Theorem}[section]
\newtheorem{proposition}[theorem]{Proposition}
\newtheorem{lemma}[theorem]{Lemma}
\newtheorem*{claim}{Claim}
\newtheorem{corollary}[theorem]{Corollary}
\newtheorem{conjecture}[theorem]{Conjecture}
\newtheorem{definition}[theorem]{Definition}
\newtheorem{problem}[theorem]{Problem}
\newtheorem{example}[theorem]{Example}
\newtheorem{question}[theorem]{Question}
\newtheorem{remark}[theorem]{Remark}

\onehalfspacing

\def\eps{\epsilon}
\def\E{\mathbb{E}}
\def\Z{\mathbb{Z}}
\def\R{\mathbb{R}}
\def\T{\mathbb{T}}
\def\C{\mathbb{C}}
\def\N{\mathbb{N}}
\def\P{\mathbb{P}}
\def\F{\mathbb{F}}
\def\Q{\mathbb{Q}}
\def\c{\mathbf{c}}
\def\f{\mathbf{f}}
\def\h{\mathbf{h}}
\def\m{\mathbf{m}}
\def\n{\mathbf{n}}
\def\r{\mathbf{r}}
\def\w{\mathbf{w}}
\def\x{\mathbf{x}}
\def\y{\mathbf{y}}
\def\z{\mathbf{z}}
\def\a{\alpha}
\def\b{\beta}
\def\g{\gamma}
\def\d{\delta}
\def\e{\epsilon}

\def\mmu{\boldsymbol{\mu}}
\def\llambda{\boldsymbol{\lambda}}
\def\oomega{\boldsymbol{\omega}}
\def\absmu{\tilde{\mu}}

\def\Lsys{\mathcal{L}}
\def\bone{\mathcal{B}_1}
\def\btwo{\mathcal{B}_2}
\def\B{\mathcal{B}}
\def\I{\mathcal{I}}
\def\X{\mathcal{X}}
\def\Y{\mathcal{Y}}
\def\CF{\mathcal{F}}
\def\K{\mathcal{K}}
\def\Zone{\mathcal{Z}_1}
\def\Ztwo{\mathcal{Z}_2}
\def\Zk{\mathcal{Z}_k}
\def\Zkmo{\mathcal{Z}_{k-1}}
\def\CL{\mathcal{CL}}

\def\Linfmu{L^{\infty}(\mu)}
\newcommand{\snorm}[1]{\lvert\!|\!| #1|\!|\!\rvert}
\newcommand{\type}[1]{^{[#1]}}

\def\ni{\noindent}
\def\iff{\;\;\Leftrightarrow\;\;}
\def\implies{\;\;\Rightarrow\;\;}
\def\mymod{\mbox{ mod }}
\def\tends{\rightarrow}
\def\maps{\rightarrow}
\def\ra{\rightarrow}
\def\seq#1#2{#1_1,\dots,#1_#2}
\def\sm#1#2{\sum_{#1=1}^#2}
\def\sp#1{\langle #1\rangle}
\def\ol{\overline}
\def\hf{\hat{f}}
\def\hQ{\hat{Q}}

\def \codim{{\rm codim}}


\begin{abstract}
In \cite{Gowers:2007tcs} we began an investigation of the following general question. Let $L_1,\dots,L_m$ be a system of linear forms in $d$ variables on $\F_p^n$, and let $A$ be a subset of $\F_p^n$ of positive density. Under what circumstances can one prove that $A$ contains roughly the same number of $m$-tuples $L_1(x_{1}, \dots, x_{d}), \dots, L_m(x_{1}, \dots, x_{d})$ with $x_{1}, \dots, x_{d} \in \F_{p}^{n}$ as a typical random set of the same density? Experience with arithmetic progressions suggests that an appropriate assumption is that $\|A-\delta\mathbf{1}\|_{U^k}$ should be small, where we have written $A$ for the characteristic function of the set $A$, $\d$ is the density of $A$, $k$ is some parameter that depends on the linear forms $L_1,\dots,L_m$, and $\|.\|_{U^k}$ is the $k$th uniformity norm. The question we investigated was how $k$ depends on $L_1,\dots,L_m$. Our main result was that there were systems of forms where $k$ could be taken to be $2$ even though there was no simple proof of this fact using the Cauchy-Schwarz inequality. Based on this result and its proof, we conjectured that uniformity of degree $k-1$ is a sufficient condition if and only if the $k$th powers of the linear forms are linearly independent.
In this paper we prove this conjecture, provided only that $p$ is sufficiently large. (It is easy to see that some such restriction is needed.) This result represents one of the first applications of the recent inverse theorem for the $U^k$ norm over $\F_p^n$ by Bergelson, Tao and Ziegler \cite{Tao:2008icg, Bergelson:2009itu}. We combine this result with some abstract arguments in order to prove that a bounded function can be expressed as a sum of polynomial phases and a part that is small in the appropriate uniformity norm. The precise form of this decomposition theorem is critical to our proof, and the theorem itself may be of independent interest.
\end{abstract}

\maketitle
\tableofcontents
\section{Introduction}

In \cite{Gowers:2007tcs} we investigated which systems of linear equations have the property that any uniform subset of $\F_{p}^{n}$ contains the ``expected'' number of solutions. By the ``expected'' number we mean the number of solutions one would expect in a random subset of the same density, and by a ``uniform subset of $\F_p^n$" we mean a set $A$ of density $\d$ such that $\|A-\delta\mathbf{1}\|_{U^2}$ is small, where $A$ is the characteristic function of $A$. More generally, we asked the same question with the $U^2$ norm replaced by any other $U^k$ norm. Note that the $U^k$ norms increase as $k$ increases, so the condition that $\|A-\delta\mathbf{1}\|_{U^k}$ is small becomes stronger, and there are more sets of linear forms for which it is sufficient.

This question arises naturally in the context of Szemer\'edi's theorem. If $x_0,\dots,x_{k-1}$ satisfy the equations $x_i-2x_{i+1}+x_{i+2}=0$ for $i=0,1,2,\dots,k-3$, then they lie in an arithmetic progression (in the sense that there exists $d\in\F_p$ such that $x_i=x_0+id$ for each $i$). It was shown in \cite{Gowers:2001szem} that if $\|A-\d\mathbf{1}\|_{U^{k-1}}$ is small, then $A$ contains roughly the number of arithmetic progressions of length $k$ that you would expect if the elements of $A$ had been selected randomly and independently with probability $\d$. (More precisely, this was shown in $\Z_N$ rather than $\F_p^n$, but the proof carries over very easily.) The proof used multiple applications of the Cauchy-Schwarz inequality. Moreover, this result is sharp, in the sense that $\|A-\d\mathbf{1}\|_{U^{k-2}}$ can be small without $A$ containing roughly the expected number of progressions of length $k$.

In their investigations of solutions of linear equations in the primes, Green and Tao \cite{Green:2006lep} worked out the most general result that could be proved using this kind of approach. Note first that by parametrizing the set of solutions to a system of linear equations one can talk equivalently about systems of linear forms. For instance, instead of the equations $x_i-2x_{i+1}+x_{i+2}=0$ for $i=0,1,2,\dots,k-3$ mentioned above one can look at the system of linear forms $x,x+y,x+2y,\dots,x+(k-1)y$. Green and Tao defined a notion of ``complexity" for a system of linear forms in $d$ variables $x_1,\dots,x_d$, and proved that for a system $L_1,\dots,L_m$ of complexity $k$ you will get roughly the expected number of images $L_{1}(x_{1}, \dots, x_{d}),\dots, L_{m}(x_{1}, \dots, x_{d})$  in $A$ provided that $\|A-\d\mathbf{1}\|_{U^{k+1}}$ is small. However, if one also works out the most general result that can be obtained by straightforwardly adapting the examples that prove that the $U^{k-1}$ norm is needed for progressions of length $k$, then a discrepancy emerges. It is easy to show that if the functions $L_1^k,\dots,L_m^k$ are linearly dependent, then there exists $A$ such that $\|A-\d\mathbf{1}\|_{U^k}$ is small but $A$ does not have roughly the expected number of solutions. However, there are systems of linear forms that have complexity $k$ while the functions $L_1^k,\dots,L_m^k$ are linearly \textit{in}dependent, and the easy arguments do not tell us how they behave. The main result of \cite{Gowers:2007tcs} was that for at least some such systems it is enough for $\|A-\d\mathbf{1}\|_{U^k}$ to be small. More specifically, we showed that there are systems of equations of complexity $2$ such that it is enough to assume that $\|A-\d\mathbf{1}\|_{U^2}$ is small, whereas a direct application of the argument of Green and Tao would require $\|A-\d\mathbf{1}\|_{U^3}$ to be small. 

To state our result in a concise way, we defined the \textit{true complexity} of a system of linear equations in $d$ variables to be the smallest $k$ with the following property. For every $\eta>0$ there exists $\e>0$ such that for every $\d\in[0,1]$ and every subset $A\subset\F_p^n$ of density $\d$, if $\|A-\d\mathbf{1}\|_{U^{k+1}}<\e$ then $p^{-nd}$ times the number of $m$-tuples $L_{1}(x_{1}, \dots, x_{d}),\dots, L_{m}(x_{1}, \dots, x_{d})$  in $A$ lies within $\eta$ of what one would expect in the random case (assuming that there are no degeneracies). To distinguish our notion of complexity from that of Green and Tao, we referred to theirs as \textit{Cauchy-Schwarz complexity}. 


\begin{theorem}\cite{Gowers:2007tcs}\label{oldresult}
Let $L_1,\dots,L_m$ be a system of linear forms in $d$ variables of Cauchy-Schwarz complexity at most 2. Suppose that the functions $L_1^2,\dots,L_m^2$ are linearly independent. Then the linear system $\seq L m$ has true complexity 1.
\end{theorem}

In the light of this result, we made the following natural conjecture.


\begin{conjecture}\cite{Gowers:2007tcs}\label{truecompconj}
The true complexity of a linear system $L_1,\dots,L_m$ is the least integer $k$ such that the forms $L_1^{k+1}, L_2^{k+1}, \dots, L_m^{k+1}$ are linearly independent. 
\end{conjecture}

A statement in ergodic theory analogous to Conjecture \ref{truecompconj} was proved by Leibman \cite{Leibman:2007odp} independently of our work in \cite{Gowers:2007tcs} and at about the same time. However, there does not appear to be a correspondence principle that would enable one to deduce Conjecture \ref{truecompconj} itself from his results.

Let us be slightly more precise about what it means for the forms $L_i^{k+1}$ to be linearly independent. A linear form $L$ on $\F_p^n$ in $d$ variables is a function of the form $L(x_1,\dots,x_d)=c_1x_1+\dots+c_dx_d$. Here, the variables $x_i$ are elements of $\F_p^n$ and the coefficients $c_i$ belong to $\F_p$. Clearly, a linear form just depends on its coefficients $(c_1,\dots,c_d)$, so we can view a system $L_1,\dots,L_m$ of linear forms on $\F_{p}^{n}$ as a system of linear forms on  $\F_{p}$, in which case they take values in $\F_p$. We say that a system of linear forms $L_1,\dots,L_m$ is \textit{degree-$k$ independent} if the functions $L_1^k,\dots,L_m^k$ are linearly independent when $L_1,\dots,L_m$ are viewed as functions from $(\F_{p})^d$ to $\F_p$. When $k=2$ we shall also call them \textit{square independent} and when $k=3$ we shall call them \textit{cube independent}.

The present paper is the second of three papers that elaborate in different ways on the main result of \cite{Gowers:2007tcs}. The first one \cite{Gowers:2009lfuI} obtains significantly improved bounds for Theorem \ref{oldresult} above, while the third \cite{Gowers:2009lfuIII} adapts the methods used in the context of $\F_p^n$ to the technically more challenging setting of $\Z_N$, also obtaining respectable bounds. The purpose of this paper is to prove Conjecture \ref{truecompconj} in $\F_{p}^{n}$, at least when $p$ is sufficiently large. (The precise condition we need is that $p$ should be larger than the Cauchy-Schwarz complexity of the system of linear forms. The reader may have noticed that we have not defined Cauchy-Schwarz complexity. That is because in this paper we do not use the notion in a detailed way: all we do is quote a lemma that uses a bound on Cauchy-Schwarz complexity as a hypothesis. The definition can be found in \cite{Gowers:2009lfuI}.) 

Let us briefly recall the structure of the proof of Theorem \ref{oldresult} in \cite{Gowers:2007tcs}. First of all, it is not hard to prove the following equivalent condition for a system to have true complexity~$k$.

\begin{lemma} \label{setstofunctions}
A system of linear forms $L_1,\dots,L_m$ in $d$ variables $x_1,\dots,x_d$ has true complexity $k$ if and only if the following statement holds. For every $\e>0$ there exists $c>0$ such that if $f:\F_p^{n}\ra\C$ is any function with $\|f\|_\infty\leq 1$ and $\|f\|_{U^{k+1}}\leq c$, and if $E$ is any non-empty subset of $\{1,2,\dots,m\}$, then 
\begin{equation*}
\left|\E_{x_1,\dots,x_d}\prod_{i\in E}f(L_i(x_1,\dots,x_d))\right|\leq\e.
\end{equation*}
\end{lemma}

In order to prove this for square-independent systems of Cauchy-Schwarz complexity 2, we decomposed the bounded function $f$ into three bounded parts $f_1+f_2+f_3$. The first part was ``quadratically structured," in a certain sense that allowed us to carry out explicit calculations in order to estimate the quantity $\E_{x_1,\dots,x_d}\prod_{i=1}^mf_1(L_i(x_1,\dots,x_d))$. The second was ``quadratically uniform," which means simply that $\|f_2\|_{U^3}$ is small. The third was small in $L_2$. To do this, we quoted a structure theorem of Green and Tao \cite{Green:2006mln}, which is a consequence of the inverse theorem for the $U^3$ norm \cite{Green:2008py}. 

When evaluating the average
\begin{equation*}
\E_{\x\in(\F_p^n)^d}\prod_{i=1}^m f(L_i(\x)),
\end{equation*}
we obtained a sum of $3^m$ terms. Because $f_1,f_2$ and $f_3$ were bounded, any term involving $f_3$ was small by the Cauchy-Schwarz inequality. The results of Green and Tao about Cauchy-Schwarz complexity guaranteed that any term involving $f_2$ was small as well. We were therefore left 
needing to estimate $\E_{x_1,\dots,x_d}\prod_{i=1}^mf_1(L_i(x_1,\dots,x_d)$, which, as we mentioned above, could be done by means of an explicit calculation. First, we observed that if $f$ is \textit{linearly} uniform, meaning that $\|f\|_{U^2}$ is small, then so is the function $f_1$ that comes out of the structure theorem of Green and Tao. We then did the calculation and discovered that if the functions $L_i^2$ were linearly independent, then this term too was small.

We gave an outline in the remarks of that paper of how we thought a proof of Conjecture \ref{truecompconj} might proceed. Given a linear system of Cauchy-Schwarz complexity $k$, we would need to be able to write a bounded function $f$ as a sum $g+h$, where $g$ has ``polynomial structure" of degree $k$ and $h$ is uniform of degree $k$. However, such a decomposition theorem necessarily requires an inverse theorem for the $U^{k+1}$ norm over $\F_p^n$, which for $k>2$ had not been proved at the time that \cite{Gowers:2007tcs} was written.

Since then, an inverse theorem for the $U^k$ norm for functions defined on $\F_p^n$, which we shall state formally in the next section, has been proved by Bergelson, Tao and Ziegler \cite{Bergelson:2009itu, Tao:2008icg}. Because of this, it has become feasible to prove Conjecture \ref{truecompconj}. However, proving the conjecture is not simply a matter of using this new theorem and straightforwardly generalizing our other arguments.  Instead, we have to do some work to formulate and develop a usable decomposition theorem. To do this, we follow a different method from the one in \cite{Gowers:2007tcs}, which we introduced in \cite{Gowers:2009lfuI}. The decomposition theorem of Green and Tao is inspired by arguments in ergodic theory and proved using averaging projections and energy-increment arguments. But for technical reasons it is not obvious how to generalize that approach to the cubic and higher-order cases. (It is not hard to obtain decompositions, but to be useful a decomposition has to have further properties: it is here that the difficulty lies.) In \cite{Gowers:2009lfuI} we used the Hahn-Banach theorem to obtain decomposition results that are more in the spirit of Fourier analysis, and that is what we shall do here. Again, the generalization is not straightforward. Perhaps the main difficulty is that the notion of the rank of a bilinear form, which is crucial to our earlier arguments, does not have an obvious analogue for multilinear forms.

In Section \ref{invdecomp}, we briefly outline the strategy for systems $L_1,\dots,L_m$ of Cauchy-Schwarz complexity 3. By the results of Green and Tao, such systems have true complexity at most 3.
This gives us two separate cases to consider when we are trying to prove that the expression $\E_{x_1,\dots,x_d}\prod_{i=1}^mf(L_i(x_1,\dots,x_d))$ is small.

In the first case, we may assume that the linear system $L_1, \dots, L_m$ is cube independent and that $f$ is highly quadratically uniform. Here we decompose $f$ as a sum $f_1+f_2+f_3$ such that $f_1$ has cubic structure, $f_2$ is small in $U^4$, and $f_3$ is small in an $L_p$-sense that we shall not specify exactly here. Because the system has Cauchy-Schwarz complexity 3, any average involving $f_2$ is negligible, boundedness allows us to deal with terms involving $f_3$, and an explicit computation of the average over the structured part uses the cube independence and quadratic uniformity of $f$. This is a straightforward generalization of the argument in \cite{Gowers:2009lfuI}. 

The second case is more complicated and already encapsulates all the difficulties that arise in the general case. Here we may assume that the system $L_1, \dots, L_m$ is square independent and that $f$ is highly linearly uniform. The difference between this case and Theorem \ref{oldresult} is that now we have the weaker hypothesis that the Cauchy-Schwarz complexity is at most 3. Briefly, this forces us to consider not just quadratically structured terms but cubically structured terms as well. We start off by decomposing $f$ as a sum $f_1+f_2+f_3$ such that $f_1$ has quadratic structure, $f_2$ is small in $U^3$, and $f_3$ is small in $L_1$, but this time we have to decompose the quadratically uniform part $f_2$ further into a sum $f_2+g_2+h_2$, where $f_2$ has cubic structure, $g_2$ is small in $U^4$ and $h_2$ is small in an $L_1$ sense. As before, any average involving $g_2$ as a factor is easily shown to be negligible by Cauchy-Schwarz. The computation involving the structured parts can be performed without too much difficulty, with the help of the fact that a system that is square independent is necessarily cube independent.

In order to get this approach to work, it is very important that the parameters involved in the error estimates should depend on each other in the right way. At various points we require the polynomial phases to have high rank (once we have decided what that means), and the uniform part to be arbitrarily small as a function of a certain type of ``complexity" of the structured part. Finally, we need the uniform part to remain bounded in $L_\infty$ while satisfying the preceding requirements. 

We expect the resulting decomposition theorems to be of independent interest. Since they have the flavour of arithmetic-regularity-type decompositions they necessarily result in tower-type bounds, even in simple cases. However, since the bound in the inverse theorem that we use in the proof is not explicit (and if made explicit in the current state, would certainly be far worse than tower type), there is not much reason to struggle to obtain better bounds. If, however, better bounds are discovered for the inverse theorem, it might be worth revisiting the arguments of this paper to try to obtain bounds more like those in \cite{Gowers:2009lfuI}.

Recently, Green, Tao and Ziegler proved a long-awaited inverse theorem for the $U^k$ norm for functions defined on $\Z_N$ (the case $k=4$ appears in \cite{Green:2009ith}), thereby raising the possibility of proving Conjecture \ref{truecompconj} in full for $\Z_N$. As we were on the point of submitting this paper, Green and Tao did indeed do this \cite{Green:2010ar}, which means that, at least from a qualitative point of view, the programme of which this paper forms a part is now complete.

\section{Inverse and decomposition theorems}\label{invdecomp}

As we outlined in the introduction, even the simplest possible generalization of Theorem \ref{oldresult} to the case where the system of linear forms is cube independent and has Cauchy-Schwarz complexity at most $3$ requires requires an inverse theorem for the $U^4$ norm.

An inverse theorem for the $U^3$ norm was proved by Green and Tao \cite{Green:2008py}  for $p>2$ and by Samorodnitsky \cite{Samorodnitsky:2007ldt}  for $p=2$. We write $\omega$ for $\exp(2\pi i/p)$.


\begin{theorem}\label{u3inverse}
Let $0<\delta\leq 1$ and let $p$ be a prime. Let $f:\F_p^n\rightarrow\C$ be a function 
with $\|f\|_\infty\leq 1$ and $\|f\|_{U^{3}}\geq\delta$. Then there
exists a quadratic polynomial $q:\F_p^n\rightarrow\F_p$ and a constant $\gamma(\delta)$ such that
\[|\E_{x \in \F_p^n} f(x)\omega^{q(x)}|\geq \gamma(\d).\]
\end{theorem}

It was conjectured that this result should hold for higher $U^k$ norms; in particular, a function that is large in the $U^{k+1}$ norm ought to correlate with a polynomial phase function of degree $k$. It was recently shown independently in \cite{Green:2007dpo} and \cite{Lovett:2008icg} that the conjecture is false in this generality. In particular, explicit counterexamples were given in the case $p=2$, and more generally when the degree $d$ of the polynomials involved exceeds the characteristic $p$ of the underlying field. However, even after these examples it was reasonable to believe that the conjecture was true whenever the characteristic $p$ was sufficiently large, and this was eventually proved by Bergelson, Tao and Ziegler \cite{Bergelson:2009itu,Tao:2008icg}.


\begin{theorem}\label{ukinverseth}
Let $0<\delta\leq 1$ and let $p$ be a prime. Let $f:\F_p^n\rightarrow\C$ be a function 
with $\|f\|_\infty\leq 1$ and $\|f\|_{U^{d+1}}\geq\delta$. Then there
exists a polynomial $\pi:\F_p^n\rightarrow\F_p$ of degree $d$ and a constant $\gamma(\delta)$ such that
\[|\E_{x \in \F_p^n} f(x)\omega^{\pi(x)}|\geq \gamma(\d),\]
provided that $p\geq d$.
\end{theorem}

In the case of low characteristic it was observed by Bergelson, Tao and Ziegler that the customary notion of a polynomial phase function, defined to be $\exp(2\pi i\pi(x)/p)$ for some polynomial $\pi$, was not appropriate. The problem is that such functions are not the most general multiplicative Freiman homomorphisms that one can define on $\F_p^n$: it turns out that there are other ones that involve $p^k$th roots of unity. By adopting a more general and more natural definition of a polynomial phase function, Bergelson, Tao and Ziegler were able to prove that even when $p<d$, a function $f$ that exhibits large $U^{d+1}$ norm correlates with a (multiplicative) polynomial phase of degree $c(d)$, where $c$ is a function of $d$. However, they did not show that $c(d)$ could be taken to equal $d$, so the modified inverse conjecture is not quite completely proved for low characteristic.

In the light of this, we shall assume in the remainder of this paper that $p$ is sufficiently large. In particular, we shall assume that $p$ exceeds the Cauchy-Schwarz complexity of the linear system that is being investigated.

Now let us turn to a general discussion of how to use inverse theorems to prove that a bounded function can be decomposed into a structured part, a uniform part and a small part. We shall use the following abstract result, which is Theorem 5.7 of \cite{Gowers:2008das}. It is a general ``arithmetic regularity lemma" of a kind that was introduced by Green \cite{Green:2005arl}, and is also closely related to Theorem 3.5 of \cite{Tao:2006qes}. However, the proof given in \cite{Gowers:2008das} is quite different: like the argument used to prove the decomposition theorem in \cite{Gowers:2009lfuI} it is based on the Hahn-Banach theorem.


\begin{theorem}\label{abstractdecomposition}
Let $\|.\|$ be a norm on $\R^n$ and let $\Phi\subset\R^n$ be a set of functions satisfying the following properties for some strictly increasing function $c:(0,1]\ra(0,1]$:
\begin{itemize}
\item $\Phi$ contains the constant function $\mathbf{1}$, $\Phi=-\Phi$, $\|\phi\|_\infty\leq 1$ for every $\phi\in\Phi$, and the linear span of $\Phi$ is $\R^n$;

\item $\sp{f,\phi}\leq 1$ for every $f$ with $\|f\|\leq 1$ and every $\phi\in\Phi$;

\item if $\|f\|_\infty\leq 1$ and $\|f\|\geq\e$, then there exists $\phi\in\Phi$ such that $\sp{f,\phi}\geq c(\e)$.
\end{itemize}
\noindent Let $\e>0$ and let $\eta:\R_+\ra\R_+$ be a strictly decreasing function. Then there is a constant $M_0$, depending only on $\e$ and the functions $c$ and $\eta$, such that every function $f\in\R^n$ that takes values in $[0,1]$ can be decomposed as a sum $f_1+f_2+f_3$, with the following properties:
\begin{itemize}
\item the functions $f_1$ and $f_1+f_3$ take values in $[0,1]$; 

\item $f_1$ is of the form $\sum_i\lambda_i\psi_i$, where $\sum_i|\lambda_i|=M\leq M_0$ and each $\psi_i$ is a product of functions in $\Phi$; 

\item $\|f_2\|\leq\eta(M)$; 

\item $\|f_3\|_2\leq\e$.
\end{itemize}
\end{theorem}

Let us make a few remarks about this theorem. Roughly speaking, we shall take $\Phi$ to be the set of polynomial phase functions of a certain degree (but there is a small technicality in that these take complex rather than real values) and $\|.\|$ will be the $U^{k+1}$ norm. The first two properties will then be easy to check, and the third is the inverse theorem. Theorem \ref{abstractdecomposition} will then tell us that that we can decompose an arbitrary function into a sum of degree-$k$ polynomial phase functions, a function with very small $U^{k+1}$ norm and a function with small $L_2$ norm. Another small technicality is that we shall be interested in functions that take values in an interval $[-C,C]$, but this again is easily dealt with.

Before we can proceed we shall need to know a little more about polynomial phase functions. Given a polynomial $\pi$ of degree $d$ we define the associated $d$-linear form $\kappa$ by the formula
\begin{equation*}
\kappa(h_1,\dots,h_d)=\sum_{\e\in\{0,1\}^d}(-1)^{d-|\e|}\pi(\e.h),
\end{equation*}
where $\e.h$ is shorthand for $\sum_i\e_ih_i$ and $|\e|$ is shorthand for $\e_1+\dots+\e_{d}$. Let us note some simple facts about $\kappa$. (These are well known but we include proofs for the convenience of the reader.)


\begin{lemma}\label{propertiesofkappa}
The function $\kappa$ just defined is a symmetric $d$-linear form on $\F_p^n$. Moreover, for every $x\in\F_p^n$ we have the equality 
\begin{equation*}
\kappa(h_1,\dots,h_d)=\sum_{\e\in\{0,1\}^d}(-1)^{d-|\e|}\pi(x+\e.h).
\end{equation*}
\end{lemma}

\begin{proof}
We prove both results by induction on $d$. The base cases both follow from the fact that if
$\pi$ is a polynomial of degree 1, then $\pi(x+a)-\pi(x)$ is a homogeneous linear function 
of $a$ that does not depend on $x$.

To prove the first statement, note first that it is clear from the definition that the value of $\kappa(h_1,\dots,h_d)$ is unaffected if one permutes the variables $h_1,\dots,h_d$. Next, let us reexpress $\kappa(h_1,\dots,h_d)$ as
\begin{equation*}
\sum_{\e\in\{0,1\}^{d-1}}(-1)^{d-1-|\e|}(\pi(\e_1h_1+\dots+\e_{d-1}h_{d-1}+h_d)-\pi(\e_1h_1+\dots+\e_{d-1}h_{d-1})).
\end{equation*}
For each fixed $h_d$ the function $x\mapsto\pi(x+h_d)-\pi(x)$ is a polynomial of degree
$d-1$ in $x$. Therefore, by the inductive hypothesis, for each fixed $h_d$ the form $\kappa(h_1,\dots,h_{d-1})$ is $(d-1)$-linear (and symmetric). By symmetry, the dependence
on $h_d$ is linear for fixed $h_1,\dots,h_{d-1}$. 

The second part is proved in a very similar way. Let us define $\kappa_x(h_1,\dots,h_d)$
to be $\sum_{\e\in\{0,1\}^d}(-1)^{d-|\e|}\pi(x+\e.h)$. Then we can reexpress $\kappa_x(h_1,\dots,h_d)$
as
\begin{equation*}
\sum_{\e\in\{0,1\}^{d-1}}(-1)^{d-1-|\e|}(\pi(x+\e_1h_1+\dots+\e_{d-1}h_{d-1}+h_d)-
\pi(x+\e_1h_1+\dots+\e_{d-1}h_{d-1})).
\end{equation*}
Again, for each fixed $h_d$ the function $x\mapsto\pi(x+h_d)-\pi(x)$ is a polynomial of degree
$d-1$ in $x$. Therefore, by induction, for each fixed $h_d$ we have the desired equality
$\kappa_x(h_1,\dots,h_d)=\kappa(h_1,\dots,h_d)$, which of course implies that the equality always
holds. 
\end{proof}

We are now in a position to evaluate the $U^{k+1}$ norm of a degree $k$-polynomial as well as its $U^{k+1}$ dual norm.


\begin{lemma}\label{uk*ofpoly}
Let $\pi:\F_p^n\ra\F_p$ be a polynomial of degree $k$ and let $g$ be the polynomial phase function $\omega^\pi$. Then $\|g\|_{U^{k+1}}=\|g\|_{U^{k+1}}^*=1$. 
\end{lemma}

\begin{proof}
Lemma \ref{propertiesofkappa} implies, amongst other things, that the identity $\sum_{\e\in\{0,1\}^{k+1}}(-1)^{|\e|}\pi(x+\e.h)=0$ holds for any $x \in \F_{p}^{n}$, any $h \in (\F_{p}^{n})^{k+1}$ and any polynomial $\pi$ of degree at most $k$.  It follows immediately that $\|g\|_{U^{k+1}}=1$. Next we turn our attention to the $U^{k}$ dual norm.

The generalized Cauchy-Schwarz inequality for the uniformity norms states that if for each $\e\in\{0,1\}^{k+1}$ we have a function $f_\e:\F_p^n\ra\C$, then 
\begin{equation*}
\left|\E_{x,h_1,\dots,h_k}\prod_{\e\in\{0,1\}^{k+1}}C^{|\e|}f_\e(x+\e.h)\right|\leq\prod_{\e\in\{0,1\}^{k+1}}
\|f_\e\|_{U^{k+1}}.
\end{equation*}
Here $C^{|\e|}$ is the operation of taking the complex conjugate $|\e|$ times. A proof of this inequality (for $\Z_N$ rather than $\F_p^n$, but the argument is identical) can be found in \cite{Gowers:2001szem}.

We apply this result with $f_0=f$ and $f_\e=g=\omega^\pi$ for every other $\e\in\{0,1\}^{k+1}$. The identity $\sum_{\e\in\{0,1\}^{k+1}}(-1)^{|\e|}\pi(x+\e.h)=0$ implies that for this choice of functions we have
\begin{equation*}
\E_{x,h_1,\dots,h_k}\prod_{\e\in\{0,1\}^{k+1}}C^{|\e|}f_\e(x+\e.h)=\E_xf(x)\ol{g(x)},
\end{equation*}
since the product $\prod_{\e\ne 0}C^{|\e|}\omega^{\pi(x+\e.h)}$ equals $\omega^{-\pi(x)}$.
Also, $\|g\|_{U^{k+1}}\leq\|g\|_\infty=1$. Therefore, we find that $\sp{f,g}\leq\|f\|_{U^{k+1}}\|g\|_{U^{k+1}}^{2^{k+1}-1}=\|f\|_{U^{k+1}}$. Since $f$ was arbitrary, it follows that $\|g\|_{U^{k+1}}^*\leq 1$, as claimed. If we take $f=g$ then the same identity implies that $\sp{f,g}=1$, which implies that $\|g\|_{U^{k+1}}^*=1$. (It is of course just the fact that $\|g\|_{U^{k+1}}^*\leq 1$ that we shall actually use.)
\end{proof}

Now we are ready to state and prove a deduction from Theorem \ref{abstractdecomposition} that will be an important tool for us later.


\begin{corollary}\label{degreekdecomposition}\label{desiredth}
Let $\e>0$ and let $\eta:\R_+\ra\R_+$ be a strictly decreasing function. Then there is a constant $M_0=M_0(\e,\eta)$ such that every function $f:\F_p^n\ra[-1,1]$ can be decomposed as a sum $f_1+f_2+f_3$ with the following properties.
\begin{itemize}
\item $f_1$ and $f_1+f_3$ take values in $[-1,1]$.
\item $f_1(x)$ is given by a sum of the form $\sum_i\lambda_i\omega^{\pi_i(x)}$, where each $\pi_i$ is a polynomial on $\F_p^n$ of degree at most $k$ and $\sum_i|\lambda_i|=M\leq M_0$.
\item $\|f_2\|_{U^{k+1}}\leq\eta(M)$.
\item $\|f_3\|_2\leq\e$.
\end{itemize}
\end{corollary}

\begin{proof}
Let $\Phi$ be the set of all functions $\pm(\omega^{\pi(x)}+\omega^{-\pi(x)})/2$ such that $\pi:\F_p^n\ra\F_p$ is a polynomial of degree at most $k$. Then $\mathbf{1}\in\Phi$, $\Phi=-\Phi$, $\|\phi\|_\infty\leq 1$ for every $\phi\in\Phi$, and the linear span of $\phi$ is $\R^{\F_p^n}$. (The last of these statements follows from the fact that $\Phi$ contains all the characters on $\F_p^n$.) Furthermore, every product of functions in $\Phi$ is a convex combination of phase functions $\omega^{\pi(x)}$ of degree at most $k$.

Let $f$ be an arbitrary function from $\F_p^n$ to $\R$. Lemma \ref{uk*ofpoly} implies that if $\phi=\pm(\omega^\pi+\omega^{-\pi})/2\in\Phi$, then $\|\phi\|_{U^{k+1}}^*\leq(\|\omega^\pi\|_{U^{k+1}}^*+\|\omega^{-\pi}\|_{U^{k+1}}^*)/2=1$, from which it follows that $\sp{f,\phi}\leq\|f\|_{U^{k+1}}$, so the second assumption of Theorem \ref{abstractdecomposition} holds with $\|.\|=\|.\|_{U^{k+1}}$.

The inverse theorem, Theorem \ref{ukinverseth}, tells us that if $\|f\|_\infty\leq 1$ and $\|f\|_{U^{k+1}}\geq\e$ then there is a polynomial phase function $\omega^{\pi}$ of degree at most $k$ such that $|\sp{f,\phi}|\geq c(\e)$. If $f$ is real, then $\sp{f,\omega^\pi}=\sp{f,\omega^{-\pi}}$, so setting $\phi=(\omega^\pi+\omega^{-\pi})/2$, we have $\phi\in\Phi$ and $|\sp{f,\phi}|\geq c(\e)$. By changing sign if necessary, we can then find $\phi\in\Phi$ such that $\sp{f,\phi}\geq c(\e)$, which proves the third assumption.

Since the hypotheses of Theorem \ref{abstractdecomposition} hold, we may deduce that if $f$ takes values in the interval $[0,1]$, then it can be decomposed as a sum $f_1+f_2+f_3$ with the properties given to us by that theorem. Since each $\phi\in\Phi$ is an average of two polynomial phase functions of degree at most $k$, this is exactly what we want apart from the fact that we are trying to prove a theorem about functions that take values in $[-1,1]$ rather than $[0,1]$. But to remedy this all we have to do is start with a function $f$ that takes values in $[-1,1]$ and apply the above argument to $(\mathbf{1}+f)/2$. Once we have expressed that as $f_1+f_2+f_3$, we know that $f=(2f_1-\mathbf{1})+2f_2+2f_3$, which is of the required form (with different constants, but that can of course be dealt with by replacing $\e$ by $\e/2$, $\eta(M)$ by $\eta(2M+1)/2$ and the output $M_0$ by $2M_0+1$ in Theorem \ref{abstractdecomposition}).
\end{proof}

\section{Basic properties of multilinear forms}

In order to be able to make use of our decomposition theorem in the preceding section, we need to establish some basic properties of polynomial phase functions. In particular, we must develop a useful definition of the rank of a polynomial phase function.

For a quadratic phase function $\omega^{q(x)}$ there is a standard way of proceeding: one defines a bilinear form $\b(x,y)=q(x+y)-q(x)-q(y)+q(0)$ on $\F_p^n$, and then one takes the rank of that bilinear form. (We put in $q(0)$ so that we do not have to assume that $q$ is homogeneous.)

As we commented earlier, this is less straightforward for higher-degree polynomials, for the simple reason that there is no single obviously best definition of the rank of a multilinear form. Several definitions have been considered in the literature, and they have different advantages and disadvantages. In this paper, we sidestep the problem as follows. In the quadratic case, we made use of the following lemma, which we briefly state and prove to help with the discussion.


\begin{lemma}\label{quadrank}
Let $q$ be a quadratic form on $\F_p^n$ of rank $r$. Then $|\E_x\omega^{q(x)}|\leq p^{-r/2}$.
\end{lemma}

\begin{proof}
We use a simple and standard technique for estimating Gauss sums.
\[|\E_x\omega^{q(x)}|^2=\E_{x,y}\omega^{q(x)-q(y)}=\E_{x,u}\omega^{q(x)-q(x+u)}=\E_{x,u}\omega^{-\b(x,u)-q(u)+q(0)},\]
where $\b$ is the bilinear form associated with $q$. For any fixed
$u$, the expectation $\E_x\omega^{\b(x,u)}$ is 0 unless $\b(x,u)=0$
for every $x$, in which case it is 1. But the space of $u$ such that 
$\b(x,u)=0$ for every $x$ has codimension equal to the rank of $\b$, 
so the density of this space is $p^{-r}$. It follows that 
$|\E_x\omega^{q(x)}|^2\leq p^{-r}$, which proves the result.
\end{proof}

In the absence of a clearly analogous definition of the rank of a multilinear form, we simply define it in such a way as to make the obvious generalization of the above proof work. (We adopt a similar strategy in \cite{Gowers:2009lfuIII} in order to deal with generalized quadratic phase functions on $\Z_N$, where no algebraic definition of rank appears to be of any use to us.) 

We simultaneously define the \textit{rank} of $\pi$ and of $\kappa$ to be
$-\log_p\E_{h_1,\dots,h_d}\omega^{\kappa(h_1,\dots,h_d)}$. Note that the quantity $\E_{h_1,\dots,h_d}\omega^{\kappa(h_1,\dots,h_d)}$ has a natural interpretation: for each $(h_2,\dots,h_d)$ the expectation over $h_1$ is $1$ if $\kappa(h_1,\dots,h_d)$ is constant as $h_1$ varies (since by multilinearity this constant must be $0$) and $0$ otherwise. Therefore, as in the case when $d=2$, we can think of $\E_{h_1,\dots,h_d}\omega^{\kappa(h_1,\dots,h_d)}$ as the density of the ``kernel" of $\kappa$. The big difference is that this ``kernel" is not a subspace of anything. Rather, it is a strange subset of $(\F_p^n)^{d-1}$ and its density does not have to be a negative integer power of $p$ (so the rank is not usually a positive integer). It is for this reason that we refer to this definition as an ``analytic" definition of rank rather than an algebraic one. 

The idea of defining rank analytically is one of the main ideas of this paper. On its own, it may not seem like much of an idea, since all we are doing is turning the conclusion of a lemma we would like to have about high-rank forms into a definition. The real point, which will become clearer later, is that we would expect to pay a heavy price for this when it comes to dealing with low-rank forms. But in fact we have ways of dealing with those as well, so we end up with reasonably clean proofs and do not have to delve too deeply into the structure of low-rank polynomials. (However, if we had needed such results, then we might well have been able to make use of a recent theorem of Green and Tao \cite{Green:2007dpo}, which says that if a polynomial phase has low rank in our sense, then it can be made out of a bounded number of polynomial phases of lower degree.) 

Let us go back to the easy task of checking that the analogue of Lemma \ref{quadrank} for higher-degree polynomials does indeed hold with our definition of rank. 


\begin{lemma}\label{genexpsum}
Let $\pi$ be a polynomial on $\F_p^n$ of degree $d$ and rank $r$. Then 
\[|\E_{x \in \F_p^n} \omega^{\pi(x)}| \leq p^{-r/(2^{d-1})}.\]
\end{lemma}

\begin{proof}
As is well known, the $U^k$ norms of a function increase with $k$. This remains true even when one allows $k$ to equal 1, in which case we define $\|f\|_{U^1}$ to be $|\E_xf(x)|$. This is in fact only a seminorm, but it is still the case that $\|f\|_{U^1}\leq\|f\|_{U^2}$. Therefore, $|\E_x\omega^{\pi(x)}|$ is at most the $U^{d-1}$ norm of $\omega^\pi$. But
\[\|\omega^\pi\|_{U^{d-1}}^{2^{d-1}}=\E_{x,h_1,\dots,h_{d-1}}\prod_{\e\in\{0,1\}^{d-1}} C^{|\e|}(\omega^{\pi(x+\e.h)})=\E_{x,h_1,\dots,h_{d-1}}\omega^{\sum_\e(-1)^{|\e|}\pi(x+\e.h)},\]
and
\begin{equation*}
\kappa(x,h_1,\dots,h_{d-1})=\sum_{\e\in\{0,1\}^{d-1}}(-1)^{d-|\e|}\pi(\e.h)-\sum_{\e\in\{0,1\}^{d-1}}(-1)^{d-|\e|}\pi(x+\e.h),
\end{equation*}
by Lemma \ref{propertiesofkappa} with $x=0$ and $(h_{1}, \dots, h_{d})$ replaced by $(x, h_{1}, \dots, h_{d-1})$. Therefore, $\sum_{\e\in\{0,1\}^{d-1}}(-1)^{d-|\e|}\pi(x+\e.h)$ is equal to $-\kappa(x,h_1,\dots,h_{d-1})$ plus a function that depends on $h_1,\dots,h_{d-1}$ only. By this fact, the remarks following the definition of rank, and the fact that the factor $(-1)^{1-d}$ does not affect the modulus in the penultimate line below,
\begin{align*}
|\E_{x,h_1,\dots,h_{d-1}}\omega^{\sum_\e(-1)^{|\e|}\pi(x+\e.h)}|
&\leq\E_{h_1,\dots,h_{d-1}}|\E_x\omega^{\sum_\e(-1)^{|\e|}\pi(x+\e.h)}|\\
&=\E_{h_1,\dots,h_{d-1}}|\E_x\omega^{\kappa(x,h_1,\dots,h_{d-1})}|\\
&=p^{-r},\\
\end{align*}
which proves the result.
\end{proof}

The final lemma in this section is a simple example of a statement that might at first appear to demand
some knowledge of the structure of low-rank polynomials (which is what we used in the
quadratic case) but that can in fact be given a straightforward analytic proof. It gives us
a very useful dichotomy for degree-$d$ phase functions: either they have small $U^d$
norm or they have not too large $(U^d)^*$ norm. Moreover, which of these is the case
depends only on the rank of the polynomial.


\begin{lemma}\label{multiuk} Let $\pi$ be a polynomial of degree $d$ and rank $r$. Then 
\[\|\omega^\pi\|_{U^d}=p^{-r/2^d} \;\; \mbox{ and } \;\; \|\omega^\pi\|_{U^d}^*=p^{r/2^d}.\]
\end{lemma}

\begin{proof}
The evaluation of the $U^d$ norm follows from Lemma \ref{propertiesofkappa}
and the definition of rank. Indeed, that lemma tells us that for all $x \in \F_p^n$ and $y \in (\F_{p}^{n})^{d}$, we have the identity
\[\kappa(y)= \sum_{\e\in \{0,1\}^d} (-1)^{d-|\e|} \pi(x+\e.y),\]
where $\kappa$ is the symmetric multilinear form associated with the polynomial $\pi$.
It follows that
\[\|\omega^{\pi}\|_{U^d}^{2^d}=\E_{y \in (\F_p^n)^d} \omega^{\kappa(y)}=p^{-r}.\]

For the dual norm, given a function $f:\F_p^n \maps \C$, let us define the nonlinear
operator $D_{2^d-1} f$ to be the function whose value at $x$ is
\[\E_{y \in (\F_p^n)^d} \prod_{\e \in \{0,1\}^d\setminus 0} \mathcal{C}^{|\e|} f(x+\e.y).\]
Now for any function $g:\F_p^n \maps \C$, the definition of rank gives us that
\[\langle g,\omega^\pi\rangle = \E_x g(x) \omega^{-\pi(x)} p^{r}\E_{y \in (\F_p^n)^d}\omega^{(-1)^{d}\kappa(y)} = p^{r}\langle g,D_{2^d-1} \omega^\pi\rangle.\]
By the generalized Cauchy-Schwarz inequality for the uniformity norms and the definition of 
$D_{2^d-1}$, we find that $|\langle g,\omega^{\pi}\rangle|$ is bounded above by $p^{r} \|g\|_{U^d} \|\omega^\pi\|_{U^d}^{2^d-1}$. By the first part of the lemma, $|\langle g,\omega^{\pi}\rangle|$ is at most $p^r\|g\|_{U^d}p^{-r(2^d-1)/2^d}=\|g\|_{U^d}p^{r/2^d}$. It follows that $\|\omega^\pi\|_{U^d}^* \leq p^{r/2^d}$.
\end{proof}

\section{A decomposition into high-rank polynomial phase functions}

In this section we shall prove a decomposition theorem that is similar to 
Corollary \ref{degreekdecomposition}, but with two important differences.
The first is that we shall split a function up into polynomial phase functions
that do not all have the same degree. The second, which is central to our
entire argument, is that we need the ranks of these polynomial phase
functions to be large. Precisely how large is a complicated matter: we have
a series of parameters and it is essential to understand how they depend 
on each other when it comes to applying the theorem later.

To begin with, we shall ignore the ranks, and obtain a preliminary decomposition
by simply iterating Corollary \ref{degreekdecomposition}. In the statement of
the theorem, we make a slightly artificial distinction when we discuss what 
various functions depend on. Given a function $f$ of two variables $x$ and $y$,
it is sometimes convenient to rewrite $f(x,y)$ as $f_x(y)$ and think of it as a 
function of $y$ that depends on $x$. And then, if $x$ is clear from the context,
one may even suppress the dependence on $x$ in the notation. For instance,
if one is proving a statement of the form, ``For every $x$ there is a function 
$f:\R\ra\R$ such that ...", then one could regard this as a proof of the existence
of a function $F$ of two variables ($x$ and a real number). These considerations
apply to the quantities $M_{i,0}$ below, which can be thought of as constants
that depend on several real variables, a function $\eta_i$, and a small real
parameter $\e$, or as functions of real variables that depend on $\eta_i$ and
$\e$, or as functions of several variables, some of which are large reals, one
of which is itself a function, and one of which is a small real. We highlight this
matter here, because it is very important for our proof that we do not accidentally
have a directed cycle of dependences, and it is not particularly easy to keep track
of whether we have done so.

\begin{theorem}\label{mixeddegreedecomposition}
Let $s$ and $k$ be positive integers with $s\leq k$ and let $\e>0$. For each $i$ 
from $s$ to $k$ let $\eta_i$ be a function from $\R_+^{i-s+1}$ to $\R_+$ that is strictly
decreasing in each variable. Let $f$ be a 
function on $\F_{p}^{n}$ that takes values in the interval $[-1,1]$. Then there are functions
$M_{s,0},\dots,M_{k,0}$, with $M_{i,0}:\R_+^{i-s}\ra\R_+$ a function that is
increasing in each variable (and that depends on $\e$ and the function 
$\eta_i$), and a decomposition
\begin{equation*}
f=f_s+\dots+f_k+g_k+h_s+\dots+h_k
\end{equation*}
with the following properties for every $i$ between $s$ and $k$.
\begin{itemize}
\item We can write $f_i=\sum_j\lambda_{i,j}\omega^{\pi_{i,j}}$,
where the functions $\pi_{i,j}$ are polynomials of degree $i$ and the 
$\lambda_{i,j}$ are real coefficients with $\sum_j|\lambda_{i,j}|=M_i\leq M_{i,0}(M_s,\dots,M_{i-1})$.
\item Let $g_i=f_{i+1}+\dots+f_k+g_k+h_{i+1}+\dots+h_k$. Then
$\|g_i\|_{U^{i+1}}\leq\eta_i(M_s,\dots,M_i)$ for each $i$.
\item The functions $f_i$ and $f_i+h_i$ take values in $[-2^{i-s},2^{i-s}]$ and $g_i$ takes
values in $[-2^{i+1-s},2^{i+1-s}]$.
\item We have the estimate $\|h_i\|_2\leq 2^{s-i-1}\e$.
\end{itemize}
\end{theorem}

\begin{proof}
We prove this by induction on $k$. The base case, when $k=s$, is precisely
Corollary \ref{degreekdecomposition} (with $\e$ replaced by $\e/2$), with
the additional trivial observation that if
$f=f_s+g_s+h_s$ and $f$ and $f_s+h_s$ all take values in $[-1,1]$, then
$g_s$ takes values in $[-2,2]$). So now let us assume that we have the 
result for $k$ and let us prove it for $k+1$.

To do this, we simply apply Corollary \ref{degreekdecomposition} to the
function $g_k$, or more precisely to the function $2^{s-k-1}g_k$, which 
takes values in $[-1,1]$. Applying it with $k+1$ instead of $k$ and $\e$ replaced by $2^{2s-2k-3}\e$
and $\eta$ replaced by the function 
$M\mapsto2^{s-k-1}\eta_{k+1}(M_s,\dots,M_k,M)$, and then multiplying
everything by $2^{k+1-s}$, we find that we can 
write $g_k$ as $\sum_j\lambda_{k+1,j}\omega^{\pi_{k+1,j}}+g_{k+1}+h_{k+1}$,
with the following properties.
\begin{itemize}
\item $\sum_j|\lambda_{k+1,j}|=M_{k+1}$ is bounded above by a constant
$M_{k+1,0}$ that depends on $\e$ and the function 
$M\mapsto2^{s-k-1}\eta_{k+1}(M_s,\dots,M_k,M)$.
\item $\|g_{k+1}\|_{U^{k+2}}\leq\eta_{k+1}(M_s,\dots,M_k,M_{k+1})$.
\item $f_{k+1}$ and $f_{k+1}+h_{k+1}$ take values in $[-2^{k+1-s},2^{k+1-s}]$,
and therefore $g_{k+1}$ takes values in $[-2^{k+2-s},2^{k+2-s}]$.
\item $\|h_{k+1}\|_2\leq2^{s-k-2}\e$.
\end{itemize}
This almost completes the proof, but it remains to check that $M_{k+1,0}$
depends just on $M_s,\dots,M_k$, $\e$ and $\eta_{k+1}$. This is true since
it depends just on $\e$ and the function $M\mapsto 2^{s-k-1}\eta_{k+1}(M_s,\dots,M_k,M)$,
and that function depends on $M_s,\dots,M_k$ and $\eta_{k+1}$ only. 
\end{proof}

The next step is to prove a result that can be used as a tool for eliminating 
polynomial phase functions of low rank.


\begin{proposition}\label{boundedhruk}
Let $\eps>0$ and $M$ be constants, and let $\eta$ be a constant such that $0<\eta\leq\e^2/M$. Then for every positive real number $R$ there is a constant $c=c(\e,R,M)$ with the following property. 
Let $f:\F_p^n\ra\R$ be a function such that $\|f\|_{U^m}\leq c$ and suppose that we have a decomposition $f=\sum_j\lambda_j\omega^{\pi_j}+g+h$ such that the functions $\pi_j$ are polynomials of degree $m$, $\sum_j|\lambda_j|=M$, $\|g\|_{U^{m+1}}\leq\eta$, and $\|h\|_2\leq\e$. Then there is also a decomposition $f=\sum_j\lambda_j'\omega^{\pi'_j}+g+h''$ such that the $\pi_j'$ are polynomials of degree $m$ and rank at least $R$, $\sum_j|\lambda_j'|\leq M$, $\|h''\|_2\leq 5\e$, and $g$ is the same function as before. 
\end{proposition}

\begin{proof}
Our approach is a natural one: if $\|f\|_{U^m}$ is very small, then it has hardly any correlation with a low-rank degree-$m$ phase function, so we would not expect such functions to play an important role in the decomposition. And indeed, we shall show that the $L_2$ norm of the ``low-rank part" of the decomposition is small enough for us to be able to absorb that part into the $L^2$ error term $h$.

First, we need to identify the ``low-rank part". To do this, we choose $t$ such that $M^{2}p^{-t}=\e^2$, and we find a ``rank gap'' of length $t$; that is, we find a number $R_1\geq R$ such that
\[\sum\{|\lambda_i|:R_1\leq r(\pi_i)<R_1+t\}\leq\eps,\]
where we have written $r(\pi_i)$ to stand for the rank of $\pi_i$. We know that 
$\sum_i|\lambda_i|\leq M$, so we must be able to find such an $R_1$ with $R_1\leq R+tM/\eps$.

Let $L=\{i:r(\pi_i)<R_1\}$ and $H=\{i:r(\pi_i)\geq R_1+t\}$. (These letters stand for ``low" and ``high", respectively.) Then we can write
\[\sum_i\lambda_i\omega^{\pi_i}=\sum_{i\in L}\lambda_i\omega^{\pi_i}+\sum_{i\in H}\lambda_i\omega^{\pi_i}+f_M,\]
and $\|f_M\|_2\leq\|f_M\|_\infty\leq\eps$. Let $f_L=\sum_{i\in L}\lambda_i\omega^{\pi_i}$
and $f_H=\sum_{i\in H}\lambda_i\omega^{\pi_i}$, so
that $f$ has a decomposition of the form $f_L+f_H+g+h'$, where $f_L$ is
made out of functions $\omega^\pi$ with $\pi$ of rank at most $R_1$,
$f_H$ is made out of such functions with $\pi$ of rank at least $R_1+t$, 
$h'=h+f_M$ has $L_2$ norm at most $2\eps$, and $\|g\|_{U^{m+1}}\leq \eta$. 
Clearly we also have $\sum_{i\in H}|\lambda_i|\leq M$. 

We would like to show, using the hypothesis that $f$ is highly uniform of degree $m$, that $\|f_L\|_2$ is very small, so that $f_L$ can be incorporated into the $L_2$ error. To do this, let us bound $\|f_L\|_2^2=\sp{f_L,f_L}$ above by
\[|\langle f_L,f \rangle|+|\langle f_L,f_H\rangle|+|\langle f_L,g \rangle|+|\langle f_L,h' \rangle|.\]
and consider each of the terms on the right-hand side in turn. 

First, we bound $|\langle f_L,f \rangle|$ above by $\|f_L\|_{U^{m}}^* \|f\|_{U^{m}}$. But by Lemma \ref{multiuk}, $\|f_L\|_{U^{m}}^*\leq Mp^{R_1}$, so we must choose $c$ to satisfy $cMp^{R_1}\leq cMp^{R+tM/\e}\leq\eps^2$, then $|\sp{f_L,f}|\leq\e^2$. Note that $t$ was chosen in terms of $M$ and $\eps$, so $c$ will be bounded in terms of $R$, $M$ and $\e$.

Next, we consider $|\langle f_L,f_H \rangle|\leq \|f_L\|_{U^{m}}^* \|f_H\|_{U^{m}}$ and use the fact that $\|f_H\|_{U^m} \leq Mp^{-(R_1+t)}$, again by Lemma \ref{multiuk}. Since $\|f_L\|_{U^m}^*\leq Mp^{R_1}$, this gives us the bound $|\sp{f_L,f_H}|\leq M^2p^{R_1-(R_1+t)}=M^2p^{-t}$, which is at most $\e^2$ by our choice of $t$.

The next term, $|\sp{f_L,g}|$, is bounded above by $ \|f_L\|_{U^{m+1}}^* \|g\|_{U^{m+1}}$. Since a degree-$m$ polynomial phase function has $(U^{m+1})^*$ norm 1, by Lemma \ref{uk*ofpoly}, the triangle inequality tells us that $\|f_L\|_{U^{m+1}}^*\leq M$, and the initial decomposition gave us the bound $\|g\|_{U^{m+1}}\leq\eta$. Since we have insisted that $\eta\leq\e^2/M$, we deduce that $|\sp{f_L,g}|\leq\e^2$.

Finally, we have that  $|\langle f_L,h' \rangle|\leq 2\eps \|f_L\|_2$. The upshot of all these computations is that $\|f_L\|_2^2 \leq 3\eps^2 + 2\eps \|f_L\|_2$, which implies that $\|f_L\|_2 \leq 3\eps$.

So provided that $\|f\|_{U^m} \leq c=c(\eps, M, R)$, we have successfully decomposed $f$ as
\[f=\sum_i \lambda_i \omega^{\pi_i} +g+h'',\]
where the $\pi_i$ are polynomials of degree $m$ and rank at least $R$, we have set $h''=h+f_M+f_L$, and we have the bounds $\sum_i|\lambda_i| \leq M$,  $\|g\|_{U^{m+1}} \leq \eta$, and $\|h''\|_{2} \leq 5\eps$.
\end{proof}

We now apply Proposition \ref{boundedhruk} iteratively to the decomposition
obtained in Theorem \ref{mixeddegreedecomposition} in order to make all the polynomial
phase functions have high rank and thereby prove our main theorem.


\begin{theorem} \label{maindecomposition}
Let $s$ and $k$ be positive integers with $s\leq k$, let $\e>0$, and let $\eta:\R_+^{k-s+1}\ra\R_+$
be a function that is strictly decreasing in each variable. Let $R_s,\dots,R_k$ be functions from $\R_+^{k-s+1}$ to $\R_+$ that are strictly increasing in each variable. Then there are functions 
$M_{s,0},\dots,M_{k,0}$, where $M_{i,0}$ is a function from $\R_+^{i-s}$ to $\R_+$ (that depends on $\e$, $\eta$ and the functions $R_i,\dots,R_k$) and a constant $c'=c'(\e,\eta,R_s,\dots,R_k)>0$, such that if $f$ is any function that takes values in $[-1,1]$ and satisfies $\|f\|_{U^s}\leq c'$, then there are real numbers $M_s,\dots,M_k$ and a decomposition 
\begin{equation*}
f=f'_s+\dots+f'_k+g+h
\end{equation*}
with the following properties.
\begin{itemize}
\item We can write $f_i'=\sum_j\lambda_{i,j}\omega^{\pi_{i,j}}$,
where the functions $\pi_{i,j}$ are polynomials of degree $i$ and the 
$\lambda_{i,j}$ are real coefficients with $\sum_j|\lambda_{i,j}|=M_i\leq M_{i,0}(M_s,\dots,M_{i-1})$.
\item For each $i$, each polynomial $\pi_{i,j}$ has rank at least $R_i(M_s,\dots,M_k)$.
\item $\|g\|_{U^{k+1}}\leq\eta(M_s,\dots,M_k)$.
\item $\|h\|_2\leq\e$.
\end{itemize}
\end{theorem}

\begin{proof}
We begin by applying Theorem \ref{mixeddegreedecomposition}. For that we shall need to
specify the functions $\eta_s,\dots,\eta_k$. We shall do that soon, but for now
let us simply apply it for some general functions $\eta_i$ (bearing in mind that $\eta_i$
is a function of the variables $M_s,\dots,M_i$).

Let $f_s+\dots+f_k+g_k+h_s+\dots+h_k$ be the decomposition that results. We 
begin by isolating the function $g_{k-1}=f_k+g_k+h_k$
from this, about which we know that $f_k=\sum_j\lambda_{k,j}\omega^{\pi_{k,j}}$,
where the functions $\pi_{k,j}$ are polynomials of degree $k$,
$\sum_j|\lambda_{kj}|=M_k\leq M_{k,0}(M_{s}, \dots, M_{k-1})$ and $\|g_k\|_{U^{k+1}}\leq\eta_k(M_s,\dots,M_k)$. Moreover, we know that $g_{k-1}$, $f_k$ and 
$f_k+h_k$ take values in $[-2^{k-s},2^{k-s}]$, $g_k$ takes values in 
$[-2^{k+1-s},2^{k+1-s}]$, $\|h_k\|_2\leq 2^{s-k-1}\e$ and
$\|g_{k-1}\|_{U^k}\leq\eta_{k-1}(M_s,\dots,M_{k-1})$.

We are already in a position to specify $\eta_k$: we take $\eta_k(M_s,\dots,M_k)$ to be 
the minimum of $\eta(M_s,\dots,M_k)$ and $2^{2(s-k-1)}\e^2/M_k$. We 
shall also take $g$ to be $g_k$, so we have the estimate 
$\|g\|_{U^{k+1}}\leq\eta(M_s,\dots,M_k)$, which will give us what we want provided that we do not 
increase any of the $M_i$ when we find our new high-rank decomposition.

Now let us suppose that we have chosen the functions $\eta_k,\eta_{k-1},\dots,\eta_i$.
We shall choose $\eta_{i-1}$ as follows. First, apply Proposition \ref{boundedhruk} to 
the function $g_{i-1}=f_i+g_i+h_i$
with $\e$ replaced by $2^{s-i-1}\e$ and with $R=R_i(M_s,\dots,M_{i-1},N_i,\dots,N_k)$,
where $N_i=M_{i,0}(M_s,\dots,M_{i-1})$ and $N_h=M_{h,0}(M_s,\dots,M_{i-1},N_i,\dots,N_{h-1})$
for each $h$ from $i+1$ to $k$. Note that $N_h\geq M_h$ for each $h$. (Note also that
$N_h$ has a dependence on $i$, but we are regarding $i$ as fixed and suppressing
that dependence.) That tells us that if $\|g_{i-1}\|_{U^{i}}\leq c(2^{s-i-1}\e,R, N_{i})$ and 
$\|g_{i}\|_{U^{i+1}}\leq 2^{2(s-i-1)}\e^2/M_i$, then we can split $g_{i-1}$ up as $f_i'+g_i+h_i'$, where 
$f_i'=\sum_j\lambda_{i,j}'\omega^{\pi_{i,j}'}$ for some degree-$i$ polynomials 
$\pi_{i,j}'$ of rank at least $R$, $\sum_j|\lambda_{i,j}'|=M_i'\leq M_{i,0}(M_s,\dots,M_{i-1})$, 
and $\|h_i'\|_2\leq 5.2^{s-i-1}\e$. So let us choose our function $\eta_{i-1}$ in such a way that
$\eta_{i-1}(M_s,\dots,M_{i-1})\leq c(2^{s-i-1}\e,R, N_{i})$, and in order to get the next stage to
work, let us also insist that $\eta_{i-1}(M_s,\dots,M_{i-1})\leq 2^{2(s-i)}\e^2/M_{i-1}$.

We are now ready to choose our constant $c'$. This we do by simply continuing the above procedure for one more step. That is, we think of $f$ as $g_{s-1}$, and we define a ``function'' $\eta_{s-1}$ by setting $i=s$ in the above paragraph. However, $\eta_{s-1}$ no longer depends on any variables, which is what we want, since we are trying to define a constant. To be slightly more explicit, we define $R$ to be $R_s(N_s,\dots,N_k)$, where the $N_h$ are defined as above (with $i=s$), and we choose $c'$ to be $c(\e/2,R)$.

We should point out that we have very carefully (and only
just) avoided a circular dependence of parameters in the previous two paragraphs: we chose
the function $\eta_{i-1}$ to be bounded above by a function of $\e$ and $R$;
$R$ in turn depends on $M_s,\dots,M_{i-1}$ and $N_i,\dots,N_k$;  but $N_i,\dots,N_k$ 
depend on $M_s,\dots,M_{i-1}$, $\e$, and the functions $\eta_i,\dots,\eta_k$, which 
we have already chosen. Thus, once we know $M_s,\dots,M_{i-1}$, $\e$ and the 
functions $\eta_i,\dots,\eta_k$, we can determine $N_i,\dots,N_s$, then $R$, and 
finally $\eta_{i-1}(M_s,\dots,M_{i-1})$.

We are almost finished. Proposition \ref{boundedhruk} guarantees that $M_i'\leq M_i$
for each $i$. Since the functions $R_i$ are increasing in each variable,
we have guaranteed that the rank of each polynomial $\pi'_{ij}$ is at least
$R_i(M_s',\dots,M_k')$, as required. Similarly, $\|g\|_{U^{k+1}}\leq\eta(M_s',\dots,M_k')$.
Finally, setting $h=h_s+\dots+h_k$, we have $\|h\|_2\leq5\sum_{i=s}^k2^{i-s-1}\e<5.2^{k-s}\e$.
Obviously we can get rid of the factor $5.2^{k-s}$ by applying the above argument with $\e$
replaced by $\e/5.2^{k-s}$.
\end{proof}

In our application, we shall actually use a slightly simpler statement that follows immediately from the previous theorem.


\begin{corollary}\label{simplerdecomposition}
Let $s$ and $k$ be positive integers with $s\leq k$, let $\e>0$, and let $\eta:\R_+^{k-s+1}\ra\R_+$ be a function that is strictly decreasing in each variable. Let $R$ be a function from $\R_+$ to $\R_+$ that is strictly increasing in each variable. Then there is a constant $M_0$, that depends on $\e$, $\eta$ and the function $R$, and a constant $c''=c''(\e,\eta,R)>0$, such that if $f$ is any function that takes values in $[-1,1]$ and satisfies $\|f\|_{U^s}\leq c''$, then there is a real number $M$ and a decomposition $f=f'+g+h$ with the following properties.
\begin{itemize}
\item We can write $f'=\sum_j\lambda_j\omega^{\pi_j}$, where each function $\pi_j$ is a polynomial of degree between $s$ and $k$, and the $\lambda_j$ are real coefficients with $\sum_j|\lambda_j|=M\leq M_0$.
\item For each $j$, each polynomial $\pi_j$ has rank at least $R(M)$.
\item $\|g\|_{U^{k+1}}\leq \eta(M)$.
\item $\|h\|_2\leq\e$.
\end{itemize}
\end{corollary}

\begin{proof}
Let us apply Theorem \ref{maindecomposition} with all the functions $R_i$ defined by $R_i(M_s,\dots,M_k)=R(M_s+\dots+M_k)$ and with $\eta(M_{1}, \dots, M_{k})$ replaced by $\eta(M_{1}+ \dots + M_{k})$. Now define a sequence $N_s,\dots,N_k$ by taking $N_s=M_{s,0}$ (where $M_{s,0}$ is as given to us by Theorem \ref{maindecomposition}), and in general $N_{i+1}=M_{i+1,0}(N_s,\dots,N_i)$. Then for each $i$, $N_i$ is an upper bound for how large $M_{i,0}$ can possibly be. Therefore, if we take $M_0$ to be $N_s+\dots+N_k$, then we obtain the first property from the corresponding property in Theorem \ref{maindecomposition}, with $M=M_s+\dots+M_k$. The remaining three properties follow immediately from their previous counterparts.
\end{proof} 

\section{Degree-$s$ independent systems of linear forms}\label{polyind}

So far, we have made no use of the condition that we are dealing with linear forms $L_1,\dots,L_m$ that are degree-$s$ independent. Recall that a linear system $L_1,\dots,L_m$ was said to be degree-$s$ independent if the functions $L_{1}^{s}, \dots, L_{m}^{s}$ are linearly independent, where we view the linear forms $L_1,\dots,L_m$ as defined on $\F_{p}$.

In this section we shall collect together some facts that will be needed when we come to apply Theorem \ref{maindecomposition} in order to estimate expressions of the form $\E_x\prod_{i=1}^mf(L_i(x))$. We begin by showing that if the linear forms $L_1,\dots,L_m$ are degree-$s$ independent, then they are degree-$t$ independent for all $t\geq s$ (as long as $p$ is sufficiently large). This is not a surprising observation, but it will be very important to us later.

Note that in the next lemma that our linear forms are functions from $(\F_p)^d$ to $\F_p$ and that $x_1,\dots,x_s$ are elements of $(\F_p)^d$.


\begin{lemma} \label{indequiv}
Let $s$ be a positive integer and let $L_1,\dots,L_m$ be linear forms in $d$ variables that take values in $\F_p$. Suppose also that $p>s$. Then the degree-$s$ forms $L_i^s$ are linearly independent if and only if the $s$-linear forms $(x_1,\dots,x_s)\mapsto L_i({x}_1)\dots L_i({x}_s)$ are linearly independent.
\end{lemma}

\begin{proof}
Let $a=(a_1,\dots,a_s)$ be an $s$-tuple of elements of $\F_p$. We shall use the identity $\sum_{\e\in\{0,1\}^s}(-1)^{s-|\e|}(\e.a)^s=s!a_1\dots a_s$, which can easily be proved by induction (in a similar manner to some of the results in Section 3 of this paper). 

Suppose, then, that $\sum_i\lambda_iL_i(x)^s=0$ for every $x \in (\F_{p})^d$. Then if we choose elements ${x}_1,\dots,{x}_s$ of $(\F_p)^d$, we know that $\sum_i\lambda_i(L_i\sum_j\e_j{x_j})^s=0$ for every $\e\in\{0,1\}^s$. Using the linearity of the $L_i$ and the identity, we deduce that 
\begin{align*}
s!\sum_i\lambda_iL_i({x}_1)\dots L_i({x}_s)&=
\sum_i\lambda_i\sum_{\e\in\{0,1\}^s}(-1)^{s-|\e|}\Bigl(\sum_j\e_jL_i({x}_j)\Bigr)^s\\
&=\sum_{\e\in\{0,1\}^s}(-1)^{s-|\e|}\sum_i\lambda_i\Bigl(L_i\sum_j\e_j{x}_j\Bigr)^s\\
&=0.\\
\end{align*}

Since $p>s$, we know that $s!\ne 0$, so if the $s$-linear forms $(L_i({x}_1)\dots L_i({x}_s))_{i=1}^m$ are linearly independent, then all the $\lambda_i$ must be $0$. This implies that the functions $L_i^s$ are linearly independent.

The other direction is trivial, since $L_i({x})^s$ is just $L_i({x}_1)\dots L_i({x}_s)$ with all the ${x}_i$ equal to ${x}$.
\end{proof}


\begin{lemma}\label{indup}
Let $s$ be a positive integer. If a system $L_1,\dots,L_m$ of linear forms is degree-$s$ independent, then it is degree-$t$ independent for all integers $t$ such that $s\leq t<p$.
\end{lemma}

\begin{proof}
By Lemma \ref{indequiv} it is enough to prove the result for the $s$-linear and $t$-linear forms defined there instead. So let us suppose that we have $\lambda_1,\dots,\lambda_r$ such that 
\begin{equation*}
\sum_{i=1}^m\lambda_iL_i({x}_1)\dots L_i({x}_t)=0
\end{equation*}
for every ${x}_1,\dots,{x}_t \in \F_{p}$. If the $\lambda_i$ are not all $0$ then we can find $x \in \F_{p}$ such that not all the $\lambda_iL_i({x})$ are zero. For such an ${x}$, let $\mu_i=\lambda_iL_i(x)^{t-s}$ for each $i$ and observe that the $\mu_i$ are not all zero. Then 
\begin{equation*}
\sum_{i=1}^m\mu_iL_i(x_1)\dots L_i(x_s)=
\sum_{i=1}^m\lambda_iL_i(x_1)\dots L_i(x_s)L_i(x)^{t-s}=0
\end{equation*}
for every $x_1,\dots,x_s \in \F_{p}$. This is a contradiction if the $s$-linear forms are linearly independent, so the lemma is proved.
\end{proof}

Before we make use of degree-$s$ independence, we need to prove some more lemmas about the behaviour of multilinear forms, this time under the additional assumption that they have high rank. We first need to establish that $\omega^{\kappa}$ behaves like a quasirandom a function whenever $\kappa$ is a high-rank symmetric multilinear form. Before we do this, we prove a simple lemma which we will use in the proof.


\begin{lemma} \label{rankofrestriction}
Let $d\geq 2$ and let $\kappa$ be a homogeneous $d$-linear form on $\F_p^n$ of rank $r$. For each $x_d$ let $r(x_d)$ be the rank of the $(d-1)$-linear form $(x_1,\dots,x_{d-1})\mapsto\kappa(x_1,\dots,x_{d-1},x_d)$. Then $p^{-r}=\E_{x_d}p^{-r(x_d)}$.
\end{lemma}

\begin{proof}
Recall that if $r$ is the rank of a homogeneous $d$-linear form $\kappa$, then $p^{-r}$ is equal to the density of the set of $(x_2,\dots,x_d)$ such that $\kappa(x,x_2,\dots,x_d)=0$ for every $x$. The result follows immediately, provided that when $d=2$ we interpret the rank $r$ of a 1-linear form to be $0$ if it is identically zero and $\infty$ otherwise (so that $p^{-r}$ is $1$ or $0$, respectively).
\end{proof}


\begin{lemma}\label{multiqr}
Let $\kappa(x_1,\dots,x_d)$ be a symmetric $d$-linear form on $\F_p^n$ of rank at least $r$. For each $I \subsetneq [d]$, let $f_I$ be a function on $(\F_p^n)^{d}$ that depends only on those $x_i$ with $i\in I$, and suppose that $\|f_I\|_\infty$ is at most 1. Then 
\[\Bigl|\E_{x \in (\F_p^n)^d}\;\omega^{\kappa(x)}\prod_{I\subsetneq [d]} f_I(x) \Bigr|\leq p^{-r/2^{d-1}}.\]
\end{lemma}

\begin{proof}
The proof is a standard application of the Cauchy-Schwarz inequality. For $x \in (\F_p^n)^d$ and any proper subset $I \subset [s]$, denote the $|I|$-tuple $(x_i)_{i \in I}$ by $x_I$. Note that the functions $f_J$ below take variables indexed by $J$ only and are allowed to change from line to line. 

We shall proceed by induction on $d$. The case $d=1$ follows from the definition of rank (and the case $d=2$ was proved in \cite{Gowers:2009lfuI}), so let us assume that the result is true for $d-1$.
Fix an index $i \in [d]$. Without loss of generality we may assume that this index is $d$. By the triangle and the Cauchy-Schwarz inequality we have the bound
\[\left|\E_{x \in (\F_p^n)^d}\;\omega^{\kappa(x)}\prod_{I\subsetneq [d]} f_I(x) \right|^{2}\leq\E_{x_{[d-1]} \in (\F_p^n)^{d-1}}\;\left|\E_{x_d \in \F_p^n} \omega^{\kappa(x)}\prod_{J\subsetneq [d-1]} f_{J\cup\{d\}}(x)\right|^2,\]
and expanding out the inner square yields
\[\E_{x_{[d-1]} \in (\F_p^n)^{d-1}}\;\E_{x_d,x_d' \in \F_p^n} \omega^{\kappa(x_{[d-1]},x_d-x_d')}\prod_{J\subsetneq [d-1]} f_{J\cup\{d\}}(x_{J},x_d)f_{J\cup\{d\}}(x_{J},x_d').\]
Let us now write $g_{J,x_d,x_d'}(x)=f_{J\cup\{d\}}(x_{J},x_d)f_{J\cup\{d\}}(x_{J},x_d')$. What we have shown is that  
\[\left|\E_{x \in (\F_p^n)^d}\;\omega^{\kappa(x)}\prod_{I\subsetneq [d]} f_I(x) \right|^{2}\leq \E_{x_d,x_d' \in \F_p^n}\left|\E_{x_{[d-1]} \in (\F_p^n)^{d-1}}\;\omega^{\kappa(x_{[d-1]},x_d-x_d')}\prod_{J\subsetneq [d-1]}g_{J,x_d,x_d'}(x)\right|.\]
Now for each $x_d,x_d'$ the function $x_{[d-1]}\mapsto\kappa(x_{[d-1]},x_d-x_d')$ is a $(d-1)$-linear form of rank $r(x_d-x_d')$. By the inductive hypothesis, the inner expectation has modulus bounded above by $p^{-r(x_d-x_d')/2^{d-2}}$. Therefore, the right-hand side is bounded above by
\begin{equation*}
\E_{x_d,x_d'}p^{-r(x_d-x_d')/2^{d-2}}=\E_{x_d}p^{-r(x_d)/2^{d-2}}\leq(\E_{x_d}p^{-r(x_d)})^{1/2^{d-2}}=p^{-r/2^{d-2}},
\end{equation*}
where for the last equality we used Lemma \ref{rankofrestriction}. The result follows on taking square roots.
\end{proof}

We now turn to a simultaneous generalization of Lemma \ref{multiqr} and Lemma \ref{genexpsum}. Lemma \ref{genexpsum} is about the behaviour of polynomials $\pi:\F_p^n\ra\F_p$ of degree $d$, while Lemma \ref{multiqr} is about $d$-linear functions $\kappa:(\F_p^n)^d\ra\F_p$. If we just consider homogeneous polynomials, then these are at opposite ends of a spectrum of monomials of degree $d$: the polynomials $\pi$ involve the smallest possible number of variables and the $d$-linear functions involve the largest possible number (1 and $d$, respectively). Now we want to look at the cases in between. For example, if $\kappa$ is trilinear, then we will want to look at functions of the form $\kappa(x,x,x)$, which is a general homogeneous cubic, $\kappa(x,y,z)$, which is trilinear, and also the intermediate case $\kappa(x,x,y)$, which depends quadratically on $x$ and linearly on $y$. 

It will help to have a way of representing a general polynomial in $d$ variables that range over $\F_p^n$. Let us start with monomials. If our $d$ variables are $x_1,\dots,x_d$, then a monomial of degree $s$ is obtained by taking an $s$-tuple $(i_1,\dots,i_s)\in[d]^s$ with $i_1\leq\dots\leq i_s$ and defining a function $\mu(x_1,\dots,x_d)=\kappa(x_{i_1},x_{i_2},\dots,x_{i_s})$, where $\kappa=\kappa(u_1,\dots,u_s)$ is some $s$-linear form. Moreover, provided that $p>s$ we shall assume that if $i_r=\dots=i_t$ then $\kappa$ is symmetric in the variables $u_r,\dots,u_t$, since we can just average over all permutations of $u_r,\dots,u_t$. This means $\kappa$ is the unique $s$-linear form giving rise to the monomial $\mu$. We define the \textit{rank} of $\mu$ to be the rank of $\kappa$. A polynomial of degree $s$ in the variables $x_1,\dots,x_d$ (each of which takes values in $\F_p^n$) is defined to be a sum of monomials of degree at most $s$, at least one of which has degree $s$.

A sequence $(i_1,\dots,i_s)\in[d]^s$ with $i_1\leq\dots\leq i_s$ can be thought of as a \textit{multisubset} of $[d]$ of \textit{size} $s$. If this multiset is $V$, then we shall write $\dot V$ for the underlying set $\{i_1,\dots,i_s\}$ (which in general will have cardinality less than $s$ since not all of $i_1,\dots,i_s$ will be distinct). If $\mu(x_1,\dots,x_d)=\kappa(x_{i_1},x_{i_2},\dots,x_{i_s})$, then we shall say that $V=(i_1,\dots,i_s)$ is the \textit{index} of $\mu$. The \textit{multiplicity} of an element $j\in\dot V$, which we shall also refer to as an element of $V$, will be defined to be the number of $h$ such that $i_h=j$, and we shall write $|V|$ for the size of $V$ (which we have defined to be $s$, which is the sum of the multiplicities of the elements of $\dot V$). If $x=(x_1,\dots,x_d)$ is a $d$-tuple of elements of $\F_p^n$, then we shall write $x_V$ for the $|V|$-tuple $(x_{i_1},\dots,x_{i_s})$.

If $f$ is any function from $(\F_p^n)^d$ to $\F_p$, $i\in[d]$, and $y\in\F_p^n$, we write $ye_i$ for the element of $(\F_p^n)^d$ which is $y$ in the $i$th place and zero everywhere else, and we write $\partial_{y,i}f$ for the function $x\mapsto f(x)-f(x-ye_i)$. Finally, if $V$ is a multisubset of $[d]$ and $i\in[d]$, then we write $V\setminus\{i\}$ for the multisubset $W$ that is the same as $V$ except that if $i$ has non-zero multiplicity $a$ in $V$ then it has multiplicity $a-1$ in $W$. For example, if $V=(1,2,2,4)$, then $V\setminus\{2\}=(1,2,4)$ and $V\setminus\{3\}=(1,2,2,4)$. We shall also write $U\subset V$ if the multiplicity of every element of $U$ is at most its multiplicity in $V$. So for example, the multisubsets of $(1,2,2)$ are $()$, $(1)$, $(2)$, $(1,2)$, $(2,2)$ and $(1,2,2)$.


\begin{lemma}\label{partialderiv}
Let $d$ and $s$ be positive integers, let $V=(i_1,\dots,i_s)$ be a multisubset $V$ of $[d]$ of size $s$, and let $\kappa_V$ be a $|V|$-linear function of rank $r$. For each $x=(x_1,\dots,x_d)\in(\F_p^n)^d$ define $\mu_V(x)$ to be the monomial $\kappa_V(x_{i_1},\dots,x_{i_s})$. Then for any fixed $y$ the function $\partial_{y,i}\mu_V$ is a polynomial made up of monomials $\nu_W$ of index $W$ with $W\subset V\setminus\{i\}$. Moreover, if $r(y)$ is the rank of the monomial $\nu_{V\setminus\{i\}}$ in this polynomial, then $\E_yp^{-r(y)}=p^{-r}$.
\end{lemma}

\begin{proof}
This we can prove by a direct calculation. For ease of notation, we shall prove it just in the case $i=d$ but of course the same argument works for general $i$. If $d\notin V$, then $\mu_V$ does not depend on $x_d$ and $\partial_{y,d}\mu_V=0$, so the result is trivial. Otherwise, let us suppose that $d$ belongs to $V$ with multiplicity $t$. Then we have
\begin{align*}
\partial_{y,d}\mu_V(x_1,\dots,x_d)&=\mu_V(x_1,\dots,x_d)-\mu_V(x_1,\dots,x_d-y)\\
&=\kappa_V(x_{i_1},\dots,x_{i_{s-t}},x_d,\dots,x_d)-\kappa_V(x_{i_1},\dots,x_{i_{s-t}},x_d-y,\dots,x_d-y).
\end{align*}
If we expand out this last expression using multilinearity, then we have a linear combination of terms of the form $\kappa_V(x_{i_1},\dots,x_{i_{s-t}},u_1,\dots,u_t)$, where each $u_j$ is equal to either $x_d$ or $y$. The term where every $u_j$ is equal to $x_d$ has a coefficient of zero, and the other terms are the values of monomials of index $W$ with $W\subset V\setminus\{d\}$. This proves the first part of the lemma.

We now turn to the assertion about ranks. Since $\kappa$ is symmetric, we have the formula
\begin{equation*}
\nu_{V\setminus\{d\}}(x_1,\dots,x_d)=t\kappa_V(x_{i_1},\dots,x_{i_{s-t}},x_d,\dots,x_d,y),
\end{equation*}
where $x_d$ is repeated $t-1$ times. The right-hand side is equal to the value of the
$(s-1)$-linear form $\lambda_y:(u_1,\dots,u_{s-1})\mapsto t\kappa_V(u_1,\dots,u_{s-1},y)$ at the point $(x_{i_1},\dots,x_{i_{s-t}},x_d,\dots,x_d)$. Now the forms $\lambda_y$ are non-zero multiples of the restrictions of $\kappa_V$ that are obtained by setting the final variable equal to $y$. Therefore, the assertion we wish to prove follows straight from Lemma \ref{rankofrestriction} and the fact that multiplying by a non-zero scalar does not change the rank of a multilinear form.
\end{proof}


\begin{lemma}\label{maxhashighrank}
Let $\pi=\pi(x_1,\dots,x_d)$ be a polynomial in $d$ variables and suppose that $\pi=\sum_{V\in\mathcal{V}}\mu_V$, where $\mathcal{V}$ is a collection of multisubsets of $[d]$ and each $\mu_V$ is a monomial of index $V$. Let $U$ be a maximal element of $\mathcal{V}$ (meaning that if $V\in\mathcal{V}$ and $U\subseteq V$ then $U=V$), let $s=|U|$ and let $r$ be the rank of $\mu_U$. Then $|\E_{x\in(\F_p^n)^d}\omega^{\pi(x)}|\leq p^{-r/2^{s-1}}$.
\end{lemma}

\begin{proof}
If $|\dot U|=s$, then the function $\mu_U$ is $s$-linear. (Recall that $\dot U$ is the underlying set of the multiset $U$.) In this case the result follows easily
from Lemma \ref{multiqr}. Indeed, without loss of generality $\mu_U$ depends on 
$x_1,\dots,x_s$. Then if we fix $x_{s+1},\dots,x_d$, we find that 
\begin{equation*}
\Bigl|\E_{x_1,\dots,x_s} \omega^{\sum_{V \in\mathcal{V}}\mu_{V}(x_1,\dots,x_d)}\Bigr|=
\Bigl|\E_{x_1,\dots,x_s}\omega^{\kappa_U(x_1,\dots,x_s)}\prod_{I\subset[s]}f_I(x_1,\dots,x_d)\Bigr|,
\end{equation*}
where 
\begin{equation*}
f_I(x_1,\dots,x_d)=\prod_{W\subset[d],W\cup I\in\mathcal{V}\setminus\{U\}}\omega^{\mu_{I\cup W}(x_1,\dots,x_d)}.
\end{equation*}
Since $x_{s+1},\dots,x_d$ are fixed, $f_I$ depends just on the variables $x_i$ with $i\in I$. Moreover, since $U$ is maximal, $f_I=1$ if $I=[s]$, since then there is no $W$ with $W\cup I\in\mathcal{V}$ and $W\cup I\ne U$. Therefore, Lemma \ref{multiqr} gives us an upper bound of $p^{-r/2^{s-1}}$. If we average over all possibilities for $x_{s+1},\dots,x_d$, then the result follows.

Now let us suppose that at least one element of $\dot U$ has
multiplicity greater than 1. Without loss of generality that element is $d$, so $\mu_U$ has a
nonlinear polynomial dependence on $x_d$. We shall apply the Cauchy-Schwarz
inequality in the usual way:
\[\Bigl|\E_{x \in (\F_p^n)^d} \omega^{\sum_{V \in\mathcal{V}}\mu_{V}(x)}\Bigr|^2\leq
\E_{x_{[d-1]}}\Bigl|\E_{x_d}\omega^{\sum_{V \in\mathcal{V}}\mu_{V}(x)}\Bigr|^2
=\E_{x_{[d-1]}}\E_{x_d,x_d'}\omega^{\sum_{V\in\mathcal{V}}(\mu_V(x)-\mu_V(x'))},\]
where we have written $x'$ for the $d$-tuple $(x_1,\dots,x_{d-1},x_d')$.

Now if we set $y_{d}=x_d-x_d'$, then $\mu_V(x)-\mu_V(x')=\partial_{y,d}\mu_V(x)$. Therefore, we can rewrite the last expression above as $\E_{y_{d}}\E_x\omega^{\partial_{y,d}\mu_V(x)}$. By Lemma \ref{partialderiv}, for each fixed $y$, each function $\partial_{y,d}\mu_V$ is a linear combination of monomials of index $V\setminus\{d\}$ if $d\in V$ and is $0$ otherwise.

Since $d$ has multiplicity greater than $1$, it follows that $U\setminus\{d\}$ is a maximal element of the multiset system $\{V\setminus\{d\}:V\in\mathcal{V}\}$. Indeed, if $U\setminus\{d\}\subseteq V\setminus\{d\}$, then $d$ must belong to $V$, from which it follows that $U\subseteq V$ and therefore that $U=V$, and finally that $U\setminus\{d\}=V\setminus\{d\}$. Therefore, by induction on $s$, we find that $|\E_x\omega^{\partial_{y,d}\mu_V(x)}|\leq p^{-r(y_{d})/2^{s-2}}$, where $r(y_{d})$ is as defined in Lemma \ref{partialderiv}. It follows that
\[|\E_{y_{d}}\E_x\omega^{\partial_{y,d}\mu_V(x)}|\leq\E_{y_{d}}|\E_x\omega^{\partial_{y,d}\mu_V(x)}|\leq\E_{y_{d}}p^{-r(y_{d})/2^{s-2}}\leq(\E_{y_{d}}p^{-r(y_{d})})^{1/2^{s-2}}=p^{-r/2^{s-2}}\]
by Lemma \ref{partialderiv}. Once again, the result follows on taking square roots.
\end{proof}

While the analytic definition we have chosen for the rank of a multilinear form is very convenient when it comes to evaluating exponential sums, it also has its disadvantages (as we also discovered in \cite{Gowers:2009lfuIII}). In particular, while it follows almost trivially from the algebraic definition of the rank of a quadratic form that the product of two low-rank quadratic phases again has low rank, one has to work to prove it from the analytic definition. However, it is still true, as Lemma \ref{multiranksum} below shows. The next few statements are in preparation for that result.


\begin{lemma} \label{multiidentity}
Let $\mu$ be a $d$-linear form on $\F_p^n$ and let $f(x_1,\dots,x_d)$ be defined to be $\omega^{\mu(x_1,\dots,x_d)}$. Then
\begin{equation*}
f(a_1,\dots,a_d)=\prod_{\e\in\{0,1\}^d}C^{d-|\e|}f(x_1+\e_1a_1,x_2+\e_2a_2,\dots,x_d+\e_da_d)
\end{equation*}
for every $a_1,\dots,a_d$ and $x_1,\dots,x_d$ in $\F_p^n$.
\end{lemma}

\begin{proof}
We prove this by induction on $d$. For any fixed $u$, the function that takes $(a_1,\dots,a_{d-1})$ to $f(a_1,\dots,a_{d-1},u)$ is a $(d-1)$-linear phase function. Therefore, if the result is true for $d-1$, then for any fixed $u \in \F_{p}^{n}$
\begin{equation*}
f(a_1,\dots,a_{d-1},u)=\prod_{\e\in\{0,1\}^{d-1}}C^{d-1-|\e|}f(x_1+\e_1a_1,\dots,x_{d-1}+\e_{d-1}a_{d-1},u).
\end{equation*}
The multilinearity of $\mu$ also implies that 
\begin{equation*}
f(a_1,\dots,a_{d-1},a_d)=f(a_1,\dots,a_{d-1},x+a_d)\ol{f(a_1,\dots,a_{d-1},x)}.
\end{equation*}
Applying the previous formula with $u$ equal to $x+a_d$ and $x$ and taking the product with the appropriate complex conjugation gives the result for $d$. The case $d=1$ is trivial to verify.
\end{proof}


\begin{lemma} \label{udlowerbound}
Let $f$ be a function from $(\F_p^n)^d$ to $\C$. Then $\|f\|_{U^d}\geq|\E_{x_1,\dots,x_d}f(x_1,\dots,x_d)|$.
\end{lemma}

\begin{proof}
Again, we use Cauchy-Schwarz several times. The first time, it tells us that 
\begin{align*}
|\E_{x_1,\dots,x_d}f(x_1,\dots,x_d)|^{2^d}&\leq(\E_{x_1,\dots,x_{d-1}}|\E_{x_d}f(x_1,\dots,x_d)|^2)^{2^{d-1}}\\
&=(\E_{x_1,\dots,x_{d-1}}\E_{x_d,y_d}f(x_1,\dots,x_d)\ol{f(x_1,\dots,x_{d-1},y_d)})^{2^{d-1}}\\
&\leq\E_{x_d,y_d}(\E_{x_1,\dots,x_{d-1}}f(x_1,\dots,x_d)\ol{f(x_1,\dots,x_{d-1},y_d)})^{2^{d-1}}.\\
\end{align*}
For each $x_d$ and $y_d$, let $h_{x_d,y_d}(x_1,\dots,x_{d-1})=f(x_1,\dots,x_d)\ol{f(x_1,\dots,x_{d-1},y_d)}$. Then the quantity above is equal to $\E_{x_d,y_d}\|h_{x_d,y_d}\|_{U^{d-1}}^{2^{d-1}}$, which also equals $\E_{x_d,a_d}\|h_{x_d,x_d+a_d}\|_{U^{d-1}}^{2^{d-1}}$, which, by the definition of the $U^{d-1}$ and $U^d$ norms, is equal to $\|f\|_{U^d}^{2^d}$.
\end{proof}

If $\mu$ is a multilinear form, let us define $\a(\mu)$ to be $p^{-r(\mu)}=\E_{x_1,\dots,x_d}\omega^{\mu(x_1,\dots,x_d)}$.


\begin{lemma}\label{multiranksum}
Let $\mu$ and $\nu$ be $d$-linear forms on $\F_p^{n}$. Then 
\begin{equation*}
r(\mu+\nu)\leq 2^d(r(\mu)+r(\nu)).
\end{equation*}
\end{lemma}

\begin{proof}
Let $f(x_1,\dots,x_d)=\omega^{\mu(x_1,\dots,x_d)}$ and let $g(x_1,\dots,x_d)=\omega^{\nu(x_1,\dots,x_d)}$. Then 
\begin{equation*}
\a(\mu+\nu)=\E_{a_1,\dots,a_d}f(a_1,\dots,a_d)g(a_1,\dots,a_d),
\end{equation*}
which, by Lemma \ref{multiidentity}, is equal to the expectation over $a_1,\dots,a_d$, $x_1,\dots,x_d$ and $y_1,\dots,y_d$ of
\begin{equation*}
\prod_{\e\in\{0,1\}^d}C^{d-|\e|}f(x_1+\e_1a_1,\dots,x_d+\e_da_d)
\prod_{\eta\in\{0,1\}^d}C^{d-|\eta|}g(y_1+\eta_1a_1,\dots,y_d+\eta_da_d).
\end{equation*}
For each $u_1,\dots,u_d$, define $h_{u_1,\dots,u_d}(x_1,\dots,x_d)$ to be $f(x_1,\dots,x_d)g(x_1+u_1,\dots,x_d+u_d)$. Then we can rewrite this expectation as
\begin{equation*}
\E_{a_1,\dots,a_d}\E_{x_1,\dots,x_d}\E_{u_1,\dots,u_d}\prod_{\e\in\{0,1\}^d}C^{d-|\e|}
h_{u_1,\dots,u_d}(x_1+\e_1a_1,x_2+\e_2a_2,\dots,x_d+\e_da_d),
\end{equation*}
which is equal to $\E_{u_1,\dots,u_d}\|h_{u_1,\dots,u_d}\|_{U^d}^{2^d}$.

This is at least $(\E_{u_1,\dots,u_d}\|h_{u_1,\dots,u_d}\|_{U^d})^{2^d}$, by H\"older's inequality (or $d$ applications of Cauchy-Schwarz), and by Lemma \ref{udlowerbound} that is at least 
\begin{align*}
(\E_{u_1,\dots,u_d}|\E_{x_1,\dots,x_d}h_{u_1,\dots,u_d}(x_1,\dots,x_d)|)^{2^d}&\geq
|\E_{u_1,\dots,u_d}\E_{x_1,\dots,x_d}h_{u_1,\dots,u_d}(x_1,\dots,x_d)|^{2^d}\\
&=|\E_{x_1,\dots,x_d}\E_{y_1,\dots,y_d}f(x_1,\dots,x_d)g(y_1,\dots,y_d)|^{2^d}\\
&=(\a(\mu)\a(\nu))^{2^d}.\\
\end{align*}
We have shown that $\a(\mu+\nu)\geq(\a(\mu)\a(\nu))^{2^d}$, and the result follows on taking logs.
\end{proof}


\begin{corollary}\label{rankofsum}
Let $\kappa_1,\dots,\kappa_m$ be $d$-linear forms on $\F_p^n$. Then 
\begin{equation*}
r(\kappa_1+\dots+\kappa_m)\leq(2m)^{d}(r(\kappa_1)+\dots+r(\kappa_m)).
\end{equation*}
\end{corollary}

\begin{proof}
We begin with the case where $m=2^h$. We claim that in this case we have a
stronger estimate in which the factor on the right-hand side is $m^d$ rather 
than $(2m)^d$. We prove this by induction, noting that when $h=1$ the 
statement is given to us by Lemma \ref{multiranksum}.

Suppose that we have proved it for all powers of $2$ up to $2^{h-1}$. Then by Lemma \ref{multiranksum}
\begin{equation*}
r(\kappa_1+\dots+\kappa_m)\leq 2^d(r(\kappa_1+\dots+\kappa_{m/2})+r(\kappa_{m/2+1}+\dots+\kappa_m)),
\end{equation*}
and by the inductive hypothesis applied to the two terms this is at most 
\begin{equation*}
2^d(m/2)^d(r(\kappa_1)+\dots+r(\kappa_m))=m^d(r(\kappa_1)+\dots+r(\kappa_m)),
\end{equation*}
which completes the inductive step.

In general, since the $d$-linear form that takes the value zero everywhere has rank zero, if $m$ is not a power of 2, then we can add enough copies of the zero map to make it up to the next power of 2. This does not increase $m$ by more than a factor of 2. The result is proved.
\end{proof}  

Equipped with this knowledge about the rank of a sum of multilinear forms of degree $d$, we now prove that for any set of multilinear forms of high rank at least one ``independent'' linear combination of these multilinear forms must have fairly high rank.


\begin{lemma}\label{multirankaverage}
Let $\kappa_1, \dots, \kappa_m$ be multilinear forms of degree $d$, at least one of which has rank at least $R$. Let $B$ be an invertible $m\times m$ matrix with entries $b_{ij} \in \F_p$. Then at least one of the multilinear forms $\eta_j=\sum_{i=1}^m b_{ij}\kappa_i$ has rank at least $R/(2m)^{d}$.
\end{lemma}

\begin{proof}
Let $\kappa_i$ have rank at least $R$. It follows from the assumption that $B$ is invertible that $\kappa_i$ is a linear combination of the forms $\eta_j$. Write $r_{\eta_{i}}$ for the rank of $\eta_{i}$, and let $r=\max_{i} r_{\eta_{i}}$. Then the rank of any linear combination of the $\eta_i$ is at most $(2m)^{d}r$, by Corollary \ref{rankofsum}. The result follows.
\end{proof}

Up to now we have made no mention of the linear forms $L_1, \dots, L_m$, which are of course central in this result, their crucial property being degree-$s$ independence. We shall now draw together the results proved in this section to obtain an estimate for exponential sums of a certain kind that involve degree-$s$ independent systems. Recall that our eventual aim is to obtain upper bounds for expressions of the form $|\E_{x_1,\dots,x_d}\prod_{i=1}^mf(L_i(x_1,\dots,x_d))|$. We shall do this by using Theorem \ref{maindecomposition} to decompose $f$ into a linear combination of high-rank polynomial phase functions plus some error terms. We shall show that the error terms can be ignored, so we will be left with a linear combination of terms of the form $\E_{x_1,\dots,x_d}\prod_{i=1}^m,f_i(L_i(x_1,\dots,x_d))$ to estimate, where each $f_i$ is a high-rank polynomial phase function. The next lemma gives us an upper bound for the size of such a term. 


\begin{proposition}\label{resultforpolyphases}
Let $m$, $s$ and $k$ be positive integers with $s\leq k$, and for each $i=1, \dots, m$ let $\pi_i:\F_p^n\ra\F_p$ be a polynomial of degree between $s$ and $k$ and of rank at least $R$. Let $L_1,\dots,L_m$ be linear forms in $d$ variables and suppose that they are degree-$s$ independent. Then 
\begin{equation*}
\Bigl|\E_{x\in(\F_p^n)^d}\omega^{\sum_{i=1}^m\pi_i(L_i(x))}\Bigr|\leq p^{-R/2^k(2m)^{d}}.
\end{equation*}
\end{proposition}

\begin{proof}
Let $t$ be the maximal degree of the polynomials $\pi_i$, satisfying $s\leq t\leq k$. Let us treat all of the $\pi_i$ as though they had degree $t$, with some (but not all) of them possibly having leading coefficient zero. The main difference this makes is that if $\pi_i$ is a polynomial that in fact has degree less than $t$ then we shall say that its rank is $0$ (as a degree-$t$ polynomial), because the $t$-linear form associated with it will be the zero form.

For each $i$, let $L_i$ be the linear form $L_i(x_1,\dots,x_d)=\sum_{u=1}^dc_{iu}x_u$ and let us write $\pi_i$ as $\pi_i(x)=\sum_{j=0}^t\kappa_{ij}(x,x,\dots,x)$, where $\kappa_{ij}$ is a symmetric $j$-linear form. Then
\begin{equation*}
\pi_i(L_i(x_1,\dots,x_d))=\pi_i(\sum_uc_{iu}x_u)=\sum_{j=0}^t\kappa_{ij}(\sum_uc_{iu}x_u,\dots,\sum_uc_{iu}x_u),
\end{equation*}
where the sums over $u$ are from $1$ to $d$. Expanding out this expression, we get
\begin{equation*}
\sum_{j=0}^t\sum_{u_1,\dots,u_j}c_{iu_1}\dots c_{iu_j}\kappa_{ij}(x_{u_1},\dots,x_{u_j})
\end{equation*}

We shall be interested in the degree-$t$ part of this, so let us write it as
\begin{equation*}
\sum_{u_1,\dots,u_t}c_{iu_1}\dots c_{iu_t}\kappa_{it}(x_{u_1},\dots,x_{u_t})+\rho(x_1,\dots,x_d),
\end{equation*}
where $\rho$ is a polynomial in $x_1,\dots,x_d$ of degree less than $t$. Note that some of the $\kappa_{it}$ may be zero, but at least one $\kappa_{it}$ has rank at least $R$.

Given a multisubset $V$ of $[d]$ size $t$, let $\sigma(V)$ be the set of all $U\in[d]^{t}$ that give rise to $V$ if their terms are written in increasing order. For example, if $V=(2,3,3)$ then the elements of $\sigma(V)$ are $(2,3,3)$, $(3,2,3)$ and $(3,3,2)$. If $U=(u_1,\dots,u_t)$, then let us write $c_{iU}$ for $c_{iu_1}\dots c_{iu_t}$ and $x_U$ for $(x_{u_1},\dots,x_{u_t})$. If $U$ and $U'$ belong to the same set $\sigma(V)$, then $c_{iU}=c_{iU'}$, and also, since the forms $\kappa_{it}$ are symmetric, $\kappa_{it}(x_U)=\kappa_{it}(x_{U'})$. Therefore, we can regard $c_{iU}$ and $\kappa_{it}(x_U)$ as functions of $V$ rather than of $U$ if we wish. Writing $\mathcal{V}_t$ for the set of all multisubsets of $[d]$ of size $t$, we also have
\begin{equation*}
\sum_{u_1,\dots,u_t}c_{iu_1}\dots c_{iu_t}\kappa_{it}(x_{u_1},\dots,x_{u_t})=
\sum_{V\in\mathcal{V}_t}\sum_{U\in\sigma(V)}c_{iU}\kappa_{it}(x_U).
\end{equation*}
If we now sum over $i$ we find that the degree-$t$ part of the polynomial function $(x_1,\dots,x_d)\mapsto\sum_{i=1}^m\pi_i(L_i(x_1,\dots,x_d))$ is
\begin{equation*}
\sum_{i=1}^m\sum_{V\in\mathcal{V}_t}\sum_{U\in\sigma(V)}c_{iU}\kappa_{it}(x_U)=\sum_{V\in\mathcal{V}_t}\sum_{i=1}^mc_{iV}'\kappa_{it}(x_V),
\end{equation*}
where $c_{iV}'=|\sigma(V)|c_{iU}$ for any $U\in\sigma(V)$.

We would now like to apply Lemma \ref{maxhashighrank}. For this purpose, we need at least one of the multilinear functions $\sum_{i=1}^mc_{iV}'\kappa_{it}$ to have high rank. 
\begin{claim}
At least one of the multilinear functions $\sum_{i=1}^mc_{iV}'\kappa_{it}$ has rank at least $R/(2m)^{d}$.
\end{claim}
\begin{proof}
It is here that we use the linear independence of $L_1^t,\dots,L_m^t$ (implicitly exploiting Lemma \ref{indup}). By this we mean that the $L_i^t$ are linearly independent when regarded as functions from $\F_p^d$ to $\F_p$. That is, if $z=(z_1,\dots,z_d)\in\F_p^d$, then we consider the function $z\mapsto(\sum_{u=1}^dc_{iu}z_u)^t$. This is a polynomial of degree $t$ in $d$ variables, and if $V=(v_1,\dots,v_t)$ is a multiset of size $t$, then the coefficient of $z_{v_1}\dots z_{v_t}$ is precisely $c_{iV}'$. It follows that if we define a matrix $(c_{iV}')$, where $i$ ranges from $1$ to $m$ and $V$ ranges over $\mathcal{V}_t$, then its $m$ rows are linearly independent (since they give us the coefficients of the polynomials $L_1^t,\dots,L_m^t$). 

Since row-rank equals column-rank, we can find $m$ multisets $V_1,\dots,V_m$ in $\mathcal{V}_t$ such that the columns $(c_{iV_j}')_{i=1}^m$ are linearly independent. By Lemma \ref{multirankaverage}, it follows that there exists $j$ such that the rank of the multilinear map $\sum_{i=1}^mc_{iV_j}'\kappa_{it}$ is at least $R/(2m)^{d}$, just as we wanted. This completes the proof of the claim.
\end{proof}

Since $V_j$ has maximal size, it is in particular maximal. Therefore, the result follows from Lemma \ref{maxhashighrank}.
\end{proof}

\section{Proof of our main conjecture in $\F_p^n$}\label{finproof}

Our aim in this paper was to establish Conjecture \ref{truecompconj} for all linear systems over $\F_p^n$ (provided that $p$ is not too small). In other words, we set out to prove the following result.


\begin{theorem}\label{generalktheorem}
Let $L_1,\dots,L_m$ be a system of $m$ linear forms in $d$ variables in $\F_{p}^{n}$ of Cauchy-Schwarz complexity $k\leq p$. Suppose that $L_1,\dots,L_m$ are degree-$s$ independent for some $s\leq p$. Then for every $\epsilon>0$ there exists $c>0$ with the following property. 

If $f:\F_p^n\rightarrow[-1,1]$ is such that $\|f\|_{U^{s}} \leq c$, then
\[\left|\E_{\x\in(\F_p^n)^d}\prod_{i=1}^m f(L_i(\x))\right|\leq\epsilon.\]
In other words, $L_1,\dots,L_m$ has true complexity at most $s-1$ .
\end{theorem}

As we commented at the beginning of the paper, Green and Tao proved that the true complexity of $L_1,\dots,L_m$ of Cauchy-Schwarz complexity $k$ is at most $k$, and we observed in \cite{Gowers:2007tcs} that if $L_1^s,\dots,L_m^s$ are linearly dependent then the conclusion of Theorem \ref{generalktheorem} is false. Therefore, if we choose the minimal possible $s$ above, then either $s=k+1$ and the theorem follows from the result of Green and Tao, or $s\leq k$. Thus, the assumption that $s\leq p$ is not important once we know that $k\leq p$. 

The next result is essentially what Green and Tao proved in \cite{Green:2004pca}, though the setting in that paper is rather more complicated because of the application to the primes. In any case, the proof is a sequence of applications of the Cauchy-Schwarz inequality combined with a judiciously chosen reparametrization of the linear system.


\begin{theorem}\label{gvnmod}
Let $f_1,\dots,f_m$ be functions defined on $\F_p^n$, and let $\seq L m$ be a linear system of Cauchy-Schwarz complexity $k\leq p$ consisting of $m$ forms in $d$ variables. Then
\[ \left|\E_{\x \in (\F_p^n)^d} \prod_{i=1}^m f_i(L_i(\x)) \right|\leq \min_i\|f_i\|_{U^{k+1}}\prod_{j\neq i} \|f_{j}\|_{\infty}.\]
\end{theorem}


Let us now turn to the proof of Theorem \ref{generalktheorem}. We begin with a brief description of the general strategy. We are aiming to prove an upper bound for $\Bigl|\E_{\x\in(\F_p^n)^d}\prod_{i=1}^m f(L_i(\x))\Bigr|$. Our first step is to decompose the first occurrence of $f$ using Corollary \ref{simplerdecomposition}. This allows us to write $f=f^{(1)}+g^{(1)}+h^{(1)}$, where $f^{(1)}$ is a linear combination of polynomial phase functions of degrees between $s$ and $k$ and high rank, $g^{(1)}$ is a function with very small $U^{k+1}$ norm, and $h^{(1)}$ has small $L_2$ norm. Having done this, we can rewrite the expression we are trying to estimate as
\begin{equation*}
\Bigl|\E_{\x\in(\F_p^n)^d}(f^{(1)}(L_m(\x))+g^{(1)}(L_m(\x))+h^{(1)}(L_m(\x)))\prod_{i=2}^mf(L_i(\x))\Bigr|,
\end{equation*} 
which splits into three terms that we can estimate separately.

In order to estimate the term involving $h^{(1)}$, we simply use the fact that $\|h^{(1)}\|_1\leq\|h^{(1)}\|_2$, and that $\prod_{i=2}^mf(L_i(\x))$ takes values in $[-1,1]$. To estimate the term involving $g^{(1)}$ we use Theorem \ref{gvnmod} and the upper bound on $\|g^{(1)}\|_{U^{k+1}}$. That leaves us with our original expression, except that now the first $f$ has been changed into an $f^{(1)}$. This represents a gain, in that $f^{(1)}$ is a linear combination of polynomial phase functions, which is what we want if we are to use Proposition \ref{resultforpolyphases}. However, we also lose something, since when we throw away the low-rank polynomial phases, we no longer know that $f^{(1)}$ takes values in $[-1,1]$. However, we do have an upper bound on the sum of the absolute values of the coefficients of the functions that make up $f^{(1)}$, so we do at least have \textit{some} upper bound $M$ for $\|f^{(1)}\|_\infty$. This means that we can play the same game with the second occurrence of $f$, as long as we replace $\e$ by $\e/M$. 

Thus, we shall end up decomposing $f$ in $m$ different ways, each time using Corollary \ref{simplerdecomposition}, but asking for smaller and smaller error terms. When we have done this, we can get rid of everything except the linear combinations of polynomial phases. Having chosen the right bounds to make this possible, we then make sure that the ranks of the polynomial phases are large (by assuming that $f$ has a sufficiently small $U^s$ norm to start with).

In order to make this argument precise, we begin by running it without specifying the functions that we use to ensure that the ranks are large (which we can do as our high-rank decomposition result, Corollary \ref{simplerdecomposition}, is true for arbitrary functions). We then work out what these functions have to be in order for Proposition \ref{resultforpolyphases} to give small enough bounds for the contribution from the polynomial phases to be small.

To do this, let $R^{(1)},\dots,R^{(m)}$ be functions with $R^{(i)}:\R_+^i\times\R_+\rightarrow\R_+$. (Here, $R^{(i)}$ will depend on variables $(M^{(1)},\dots,M^{(i)},\e)$. We shall think of it as a function of $M^{(i)}$ that is allowed to depend on the other variables.) Let $\eta=\e$ and apply Corollary \ref{simplerdecomposition} to write $f$ as $f^{(1)}+g^{(1)}+h^{(1)}$, where $f^{(1)}$ is a linear combination $\sum_j\lambda_j\omega^{\pi_j}$ such that $\sum_j|\lambda_j|=M^{(1)}\leq M^{(1)}_0(R^{(1)},\e)$, each $\pi_j$ has degree between $s$ and $k$ and rank at least $R^{(1)}(M^{(1)},\e)$, $\|g^{(1)}\|_{U^{k+1}}\leq\e$, and $\|h^{(1)}\|_2\leq\e$.

Because it is very important, we remark that $R^{(1)}$ is a function of $M^{(1)}$ and $\e$, and $M_0^{(1)}$ is a function of $\e$ and the \textit{function} $R^{(1)}$ rather than the value taken by that function. In other words, if we specify $R^{(1)}$ and $\e$, then we already know what $M_0^{(1)}$ is, quite independently of $M^{(1)}$. We can then find $M^{(1)}$ that is less than $M_0^{(1)}$. Thus, what looks at first like a circularity is in fact not circular at all. 

Now let us continue. Suppose that we have applied Corollary \ref{simplerdecomposition} $i-1$ times. On the $i$th occasion, we apply Corollary \ref{simplerdecomposition} again but with $\eta$ and $\e$ replaced by $\e^{(i)}=\e(M^{(1)}M^{(2)}\dots M^{(i-1)})^{-1}$. This time, the polynomial phases have coefficients with absolute values that sum to $M^{(i)}\leq M_0^{(i)}(R^{(i)},M^{(1)},\dots,M^{(i-1)},\e)$ and have rank at least $R^{(i)}(M^{(1)},\dots,M^{(i)},\e)$. We also have the estimates $\|g^{(i)}\|_{U^{k+1}}\leq\e^{(i)}$ and $\|h^{(i)}\|_2\leq\e^{(i)}$.


\begin{claim}
Let $f$ be decomposed as $f^{(i)}+g^{(i)}+h^{(i)}$ in $m$ ways as just described. Then 
\begin{equation*}
\Bigl|\E_{\x\in(\F_p^n)^d}\prod_{i=1}^mf(L_i(\x))-\E_{\x\in(\F_p^n)^d}\prod_{i=1}^mf^{(i)}(L_i(\x))\Bigr|\leq 2m\e.
\end{equation*}
\end{claim}

\begin{proof}
For each $q$, let us estimate
\begin{equation*}
\Bigl|\E_{\x\in(\F_p^n)^d}\prod_{i\leq q-1}f^{(i)}(L_i(\x))\prod_{i>q-1}f(L_i(\x))-
\E_{\x\in(\F_p^n)^d}\prod_{i\leq q}f^{(i)}(L_i(\x))\prod_{i>q}f(L_i(\x))\Bigr|,
\end{equation*}
which is equal to
\begin{equation*}
\Bigl|\E_{\x\in(\F_p^n)^d}\prod_{i<q}f^{(i)}(L_i(\x))(g^{(q)}(L_q(\x))+h^{(q)}(L_q(\x)))
\prod_{i>q}f(L_i(\x))\Bigr|.
\end{equation*}
Since $L_q(\x)$ is evenly distributed over $\F_p^n$, the contribution from the
$h^{(q)}$ term is at most $\|h^{(q)}\|_1\prod_{i<q}\|f^{(i)}\|_\infty\leq\e^{(q)}\prod_{i<q}M^{(i)}=\e$.
As for the contribution from the $g^{(q)}$ term, by Theorem \ref{gvnmod} it is at most 
$\|g^{(q)}\|_{U^{k+1}}\prod_{i<q}\|f^{(i)}\|_\infty\leq\e^{(q)}\prod_{i<q}M^{(i)}=\e$.

Since the quantity we are trying to estimate is the sum of the quantities we have just estimated, the claim follows from the triangle inequality.~\end{proof}

It remains to prove that $\E_{\x\in(\F_p^n)^d}\prod_{i=1}^mf^{(i)}(L_i(\x))$ is small. Since we have Proposition \ref{resultforpolyphases}, this is a question of making sure we choose the ranks $R^{(i)}$ appropriately. In a sense, this is a trivial matter, but it takes a small effort to check that the dependence of our various parameters is such that we really are free to choose the ranks to be as big as we need for the lemma to give us a good enough bound.

It follows immediately from Proposition \ref{resultforpolyphases} and the triangle inequality that we will be done if we can choose the functions $R^{(i)}$ in such a way that $R^{(i)}(M^{(1)},\dots,M^{(i)},\e)\geq R=R(M^{(1)},\dots,M^{(m)},\e)$, where $R$ is large enough for $p^{-R/2^k(2m)^{d}}$ to be at most $\e(M^{(1)}\dots M^{(m)})^{-1}$. The difficulty we must deal with is that $R^{(i)}$ does not depend on $M^{(i+1)},\dots,M^{(m)}$, and it looks as though it needs to.

We deal with this inductively as follows. Suppose that we have chosen the functions $R^{(m)},R^{(m-1)},\dots,R^{(i+1)}$ and are now trying to choose $R^{(i)}$. Let us define a sequence $N^{(i,i+1)},\dots,N^{(i,m)}$ as follows. We let $N^{(i,i+1)}=M_0^{(i+1)}(R^{(i+1)},M^{(1)},\dots,M^{(i)},\e)$, then $N^{(i,i+2)}=M_0^{(i+2)}(R^{(i+2)},M^{(1)},\dots,M^{(i)},N^{(i,i+1)})$, and so on. A trivial induction shows that the $N^{(i,j)}$ are upper bounds for the $M^{(j)}$ when $j>i$, and they depend just on $M^{(1)},\dots,M^{(i)}$, $\e$, and the already chosen functions $R^{(i+1)},\dots,R^{(m)}$. Therefore, we can define $R^{(i)}(M^{(1)},\dots,M^{(i)},\e)$ to be $R(M^{(1)},\dots,M^{(i)},N^{(i,i+1)},\dots,N^{(i,m)},\e)$.

The total error incurred in this argument is of course $(2m+1)\e$, but this is easily rectified by replacing $\eps$ with $\e/(2m+1)$ throughout.

\section{The off-diagonal case}

In this section, we briefly discuss a closely related question that can also be treated by our techniques. Recall that we initially set out to find the minimal $k$ with the following property: if $A$ is a subset of $\F_p^n$ of density $\d$ such that $\|A-\delta\mathbf{1}\|_{U^k}$ is sufficiently small, then the density of $\x\in(\F_p^n)^d$ such that $L_i(\x)\in A$ for $i=1,\dots,m$ is approximately $\d^m$. Lemma \ref{setstofunctions} allowed us to recast that as a question about functions: if we set $f$ to be $A-\delta\mathbf{1}$, then we know that $f$ is bounded and $\|f\|_{U^k}$ is small, and we want to be able to deduce that  $\E_{\x\in(\F_p^n)^d}\prod_{i\in E}f(L_i(\x))$ is small for every non-empty subset $E\subset\{1,2,\dots,m\}$. A necessary and sufficient condition on $k$ turned out to be that the linear forms $L_i$ were degree-$k$ independent.

What happens if we try to estimate the density of $\x$ such that $L_i(\x)\in A_i$ for $i=1,\dots,m$, where the sets $A_1,\dots,A_m$ do not have to be equal? Associated with each set $A_i$ will be its density $\d_i$, and in this case we would like to find a necessary and sufficient condition on the sequence $(k_1,\dots,k_m)$ such that if $\|A_i-\delta_i\mathbf{1}\|_{U^{k_i}}$ is sufficiently small for every $i$, then the density of $\x$ such that $L_i(\x)\in A_i$ for $i=1,\dots,m$ is approximately $\prod_{i=1}^m\d_i$. We call this the off-diagonal case of the problem.

We have not completely solved the off-diagonal case, but we do have a sufficient condition that generalizes the condition we obtained in the diagonal case in a natural way. The statement is as follows.

\begin{theorem} \label{offdiagonal}
For every $\epsilon>0$ and every sequence $(s_1,\dots,s_m)$ of positive integers there exists a constant $c>0$ with the following property. Let $L_1,\dots,L_m$ be linear forms such that for every $i\leq m$ it is impossible to write $L_i^{s_i}$ as a linear combination of the functions $L_j^{s_i}$ with $j\ne i$, and let $A_1,\dots,A_m$ be subsets of $\F_p^n$ such that $A_i$ has density $\d_i$ and $\|A_i-\d_i\mathbf{1}\|_{U^{s_i}}\leq c$ for every $i\leq m$. Then the density of $\x$ such that $L_i(\x)\in A_i$ for every $i$ differs from $\prod_{i=1}^m\d_i$ by at most $\e$.
\end{theorem}

As in the diagonal case, it is more convenient to work with a version of the result for functions that implies the sets version. 

\begin{theorem} \label{offdiagonalfunctions}
For every $\epsilon>0$ and every sequence $(s_1,\dots,s_m)$ of positive integers there exists a constant $c>0$ with the following property. Let $L_1,\dots,L_m$ be linear forms such that for every $i\leq m$ it is impossible to write $L_i^{s_i}$ as a linear combination of the functions $L_j^{s_i}$ with $j\ne i$, and let $f_1,\dots,f_m$ be functions from $\F_p^n$ to $\C$ such that $\|f_i\|_\infty\leq 1$ and $\|f_i\|_{U^{s_i}}\leq c$ for every $i\leq m$. Then
\begin{equation*}
\left|\E_{\x\in(\F_p^n)^d}\prod_{i=1}^mf_i(L_i(\x))\right|\leq\e.
\end{equation*}
\end{theorem}

The fact that a result like this ought to be true was observed independently by Hamed Hatami and Shachar Lovett, who were able to prove it in $\F_2^n$ using the methods from \cite{Gowers:2007tcs} in the cases that only required the inverse theorems for the $U^2$ and $U^3$ norms. In general, it is not too difficult to adapt the methods in the present paper to give the result in full generality.

To prove Theorem \ref{offdiagonalfunctions}, we need some slight strengthenings of some of the lemmas from \S \ref{polyind}. We begin with a lemma about matrices that we shall use instead of the statement that the row rank of a matrix is equal to its column rank.

\begin{lemma} \label{matrixlemma}
Let $A$ be an $m\times n$ matrix over a field $\F$, and suppose that it is not possible to express the $i$th row of $A$ as a linear combination of the other rows. Then the column space of $A$ contains the column vector with a 1 in the $i$th row and zeros everywhere else.
\end{lemma}

\begin{proof}
Without loss of generality $i=1$. We shall attempt to use column operations to produce a matrix that has the desired column vector as its first column. In other words, we would like a 1 in the top left-hand corner and for all the other rows to begin with 0. 

Since the first row cannot be all zero, we can do Gaussian column operations to make it 1 in the first place and 0 everywhere else. Note that even after doing these column operations it is still the case that the first row is not a linear combination of the remaining rows. Now let $B$ be the matrix obtained by deleting the first row. We will be done if we can prove that the first column of $B$ is a linear combination of the other columns. 

If it is not a linear combination of the other columns, then there must be a linear functional that vanishes on all the columns except for the first. Equivalently, there must be a linear combination of the rows of $B$ that vanishes everywhere except in the first coordinate. But from that it follows that the first row of the modified matrix $A$ is a linear combination of the rows of $B$, which contradicts our assumption.
\end{proof}

Next, we prove a generalization of Lemma \ref{multirankaverage}.

\begin{lemma} \label{strongmultirankaverage}
Let $\kappa_1, \dots, \kappa_m$ be multilinear forms of degree $d$ and suppose that there is some $r\leq m$ such that $\kappa_r$ has rank at least $R$. Let $B$ be an $m\times n$ matrix with entries $b_{ij} \in \F_p$ and suppose that the $r$th row of $B$ is not a linear combination of the other rows. Then at least one of the multilinear forms $\eta_j=\sum_{i=1}^m b_{ij}\kappa_i$ has rank at least $R/(2m)^{d}$.
\end{lemma}

\begin{proof}
By Lemma \ref{matrixlemma} we can find coefficients $c_1,\dots,c_n$ such that $\sum_jc_jb_{ij}=1$
if $i=r$ and $0$ otherwise. Furthermore, since the column vectors all live in $\F_p^m$, we can do this in such a way that at most $m$ of these coefficients are non-zero. But in that case, 
\begin{equation*}
\sum_jc_j\eta_j=\sum_i\sum_jc_jb_{ij}\kappa_i=\kappa_r,
\end{equation*}
so we have written $\kappa_r$ as a linear combination of at most $m$ of the forms $\eta_j$. If $r$ is the maximum rank of any $\eta_j$, it follows from Corollary \ref{rankofsum} that $\kappa_r$ has rank at most $(2m)^dr$. The result follows.
\end{proof}

Next, we need a generalization of Proposition \ref{resultforpolyphases}.

\begin{proposition}\label{strongresultforpolyphases}
Let $k$ and $m$ be positive integers. For each $i\leq m$ let $k_i$ be a positive integer less than or equal to $k$, and let $\pi_i:\F_p^n\ra\F_p$ be a polynomial of degree $k_i$. Suppose also that there is some $r$ such that $\pi_r$ has rank at least $R$ and $k_r$ is at least as big as every other $k_i$. Let $L_1,\dots,L_m$ be linear forms in $d$ variables and suppose that $L_r^{k_r}$ is not in the linear span of the other functions $L_i^{k_r}$. Then 
\begin{equation*}
\Bigl|\E_{x\in(\F_p^n)^d}\omega^{\sum_{i=1}^m\pi_i(L_i(x))}\Bigr|\leq p^{-R/2^k(2m)^{d}}.
\end{equation*}
\end{proposition}

\begin{proof}
We shall not give a complete proof. Instead, we shall just point out where the proof differs from the proof of Proposition \ref{resultforpolyphases}.

A very slight difference occurs at the end of the third paragraph, where instead of saying ``at least one $\kappa_{it}$ has rank at least $R$," it is now more appropriate to say that $\kappa_{rt}$ has rank at least $R$. Note also that $t=k_r$.

The main difference, however, is that when it comes to proving the claim, we shall use Lemma \ref{strongmultirankaverage} instead of Lemma \ref{multirankaverage}. We do not know that the functions $L_i^t$ are linearly independent, but we do know that $L_r^t$ is independent of the other $L_i^t$. From this it follows that if we define the matrix $(c_{iV'})$ just as before, the $r$th row will be independent of the other rows, which is what we need in order to be able to apply Lemma \ref{strongmultirankaverage}. We can now complete the proof by applying Lemma \ref{maxhashighrank}, just as before.
\end{proof}

It remains to discuss how the proof of Theorem \ref{offdiagonalfunctions} differs from the proof of Theorem \ref{gvnmod}. A superficial difference is that we are looking at $f_i(L_i(\x))$ instead of $f(L_i(\x))$. A deeper difference is that when we split $f_i$ up into polynomial phases, the degrees of these phases are between $s_i$ and $k$ rather than between $s$ and $k$. 

Exactly as in that proof, we reduce the task to proving a result in the case that the $f_i$ are polynomials of high rank. Furthermore, our assumption that each $L_i^{s_i}$ is independent of all the other $L_j^{s_i}$, which implies that $L_i^t$ is independent of all the other $L_j^t$ whenever $t\geq s_i$, guarantees that the condition for Proposition \ref{strongresultforpolyphases} holds for each of these terms. This completes the proof of Theorem \ref{offdiagonalfunctions}, and hence of Theorem \ref{offdiagonal} as well.
\bigskip

It may be that a substantially stronger result than Theorem \ref{offdiagonalfunctions} is true: it could be enough if there is just one $L_i^{s_i}$ is independent of the other $L_j^{s_i}$. The evidence for this is that it is true in the case where all the $s_i$ are equal to some $s$ and the system of linear forms has Cauchy-Schwarz complexity at most $s$. In that case all the polynomial phases in our decompositions have degree $s$, and the condition is that some $L_i^s$ is independent of the other $L_j^s$, which is enough for our argument to work because the polynomial phase used in the decomposition of $f_i$ has maximal degree amongst all the polynomial phases.

The simplest situation where the difficulty arises is if the $L_i$ have Cauchy-Schwarz complexity 3 and we know that $L_1^2$ is independent of the other $L_i^2$. We would like it to be enough if $f_1$ had a small $U^2$ norm, but to prove that we would have to decompose $f_1$ into quadratic and cubic phases, plus error terms, and we have trouble dealing with terms that involve the quadratic part of $f_1$ and cubic parts of other $f_i$. Thus, the following problem remains open.

\begin{problem}
Let $\epsilon>0$ and let $(s_1,\dots,s_m)$ be a sequence of positive integers. Does there exist a constant $c>0$ with the following property? Let $L_1,\dots,L_m$ be linear forms such that for some $i\leq m$ it is impossible to write $L_i^{s_i}$ as a linear combination of the functions $L_j^{s_i}$ with $j\ne i$, and let $f_1,\dots,f_m$ be functions from $\F_p^n$ to $\C$ such that $\|f_i\|_\infty\leq 1$ and $\|f_i\|_{U^{s_i}}\leq c$ for every $i\leq m$. Then
\begin{equation*}
\left|\E_{\x\in(\F_p^n)^d}\prod_{i=1}^mf_i(L_i(\x))\right|\leq\e.
\end{equation*}
\end{problem} 

A second piece of evidence in favour of a positive answer is that there is a fairly natural example that would show that, if true, such a result would be best possible. We briefly sketch the example.

\begin{example}
Let $(s_1,\dots,s_m)$ be a sequence of positive integers. Let $L_1,\dots,L_m$ be linear forms such that for each $i$ it is possible to write $L_i^{s_i}$ as a linear combination of the functions $L_j^{s_i}$ with $j\ne i$. Let $p$ be sufficiently large. Then for every $c>0$ there exist a positive integer $n$ and functions $f_1,\dots,f_m$ such that $\|f_i\|_{U^{s_i}}\leq c$ for every $i$ and
\begin{equation*}
\left|\E_{\x\in(\F_p^n)^d}\prod_{i=1}^mf_i(L_i(\x))\right|=1.
\end{equation*}
\end{example}

\begin{proof}
Let $\pi_s$ be the polynomial $x\mapsto\sum_{i=1}^nx_i^s$ (defined on $\F_p^n$). It can be checked that for fixed $s$ the rank of $\pi_s$ tends to zero as $n$ tends to infinity, and therefore that the $U^s$ norm of the function $\omega^{\pi_s}$ tends to zero. 

Now let us choose coefficients $c_{ij}\in\F_p$ such that for each $i$ we have $c_{ii}\ne 0$ and $\sum_jc_{ij}L_j^{s_i}=0$. The dependence assumption of the theorem guarantees that we can do this. Let $\mu_1,\dots,\mu_m$ be coefficients that we shall choose in a moment, and for each $i$ let $f_i$ be the function $f_i(x)=\omega^{\sum_j\mu_jc_{ji}\pi_{s_j}(x)}$. Note that the exponent is a linear combination of the polynomials $\pi_s$. We need $\|f_i\|_{U^{s_i}}$ to be small, which it will be if the coefficient of $\pi_{s_i}$ is non-zero. We know that $c_{ii}\ne 0$, so it is enough if $\mu_i\ne 0$ and the sum of the $\mu_jc_{ji}$ over all $j$ such that $s_j=s_i$ does not equal $-\mu_ic_{ii}$. If we choose the $\mu_i$ randomly, then an easy probabilistic argument shows that for large enough $p$ (depending on $m$ only) there is a non-zero probability that we will never have any cancellation of this kind. 

We now claim that 
\begin{equation*}
\left|\E_{\x\in(\F_p^n)^d}\prod_{i=1}^mf_i(L_i(\x))\right|=1.
\end{equation*}
To prove this, we first observe that if $\sum_j c_{ij}L_j^{s_i}=0$, then $\sum_j c_{ij}(\pi_{s_i}\circ L_j)=0$ as well. (Note that in the first equation we are thinking of $L_j$ as a function from $\F_p^d$ to $\F_p$ and in the second it is a function from $(\F_p^n)^d$ to $\F_p^n$.) To check this, one can expand out both sides. Therefore,
\begin{equation*}
\sum_{i,j}\mu_jc_{ji}\pi_{s_j}(L_i(\x))=\sum_j\mu_j\sum_ic_{ji}\pi_{s_j}(L_i(\x)),
\end{equation*}
which is zero, since the coefficients $c_{ji}$ have been chosen to make the inner sum zero for every $j$. It follows that $\prod_{i=1}^mf_i(L_i(\x))=1$ for every $\x$, which proves the theorem.
\end{proof}

Another problem that remains annoyingly open is to show that the dependence of $c$ on the other parameters in Theorems \ref{generalktheorem} and \ref{offdiagonalfunctions} cannot be too good. This would be a convincing argument that it was impossible to prove these theorems by some kind of clever transformation followed by multiple applications of the Cauchy-Schwarz inequality. We do not believe that such a proof exists, but it would be good to have more evidence for this.

We end with the following simple case of this problem.

\begin{problem}
Do there exist positive integers $s$ and $k$ and a degree-$s$ independent system of linear forms $L_1,\dots,L_m$ with the following property? For every positive real number $r$ there exists $\e>0$ and functions $f_i:\F_p^n\rightarrow\C$ such that $\|f_i\|_{U^s}\leq\e^r$ for every $i$, and yet $|\E_{\x\in(\F_p^n)^d}\prod_{i=1}^mf_i(L_i(\x))|>\e$?
\end{problem}


\end{document}